\documentclass[12pt]{amsart}

\usepackage[T1]{fontenc}
\usepackage{lmodern}

\usepackage{amscd, amssymb, amsmath, amsthm}
\usepackage{enumerate, latexsym, mathrsfs}
\usepackage{xypic}

\usepackage[pdftex,lmargin=1in,rmargin=1in,tmargin=1in,bmargin=1in]{geometry}
\usepackage[
		bookmarks=true, bookmarksopen=true,%
    bookmarksdepth=3,bookmarksopenlevel=2,%
    colorlinks=true,%
    linkcolor=blue,%
    citecolor=blue]{hyperref}
\newtheorem{theorem}{Theorem}[section]
\newtheorem{lemma}[theorem]{Lemma}
\newtheorem{proposition}[theorem]{Proposition}

\theoremstyle{definition}
\newtheorem{definition}[theorem]{Definition}

\newtheorem{example}[theorem]{Example}

\newtheorem{remark}[theorem]{Remark}

\newcommand{\Real}{{\mathbb R}}
\newcommand{\Rational}{{\mathbb Q}}
\newcommand{\Complex}{{\mathbb C}}
\newcommand{\Integral}{{\mathbb Z}}
\newcommand{\Matrix}{{\mathrm{Mat}}}
\newcommand{\Natural}{{\mathbb{N}}}

\newcommand{\ud}{{\mathrm{d}}}

\title{Degree of $L^2$--Alexander torsion
for 3--manifolds}

\author[Yi Liu]{%
        Yi Liu} 
\address{%
    Beijing International Center for Mathematical Research\\
    No.~5 Yiheyuan Road, Haidian District, Beijing 100871, China P.R.} 
\email{%
    liuyi@math.pku.edu.cn}  


\date{%
 \today}

\begin{document}

\begin{abstract}
	For an irreducible orientable compact $3$-manifold $N$ with empty or incompressible toral boundary,
	the full $L^2$--Alexander torsion $\tau^{(2)}(N,\phi)(t)$
	associated to any real first cohomology class $\phi$ of $N$ is represented by
	a function of a positive real variable $t$.
	The paper shows that $\tau^{(2)}(N,\phi)$ is continuous, everywhere positive, and
	asymptotically monomial in both ends. Moreover,
	the degree of $\tau^{(2)}(N,\phi)$ equals the Thurston norm of $\phi$. The result
	confirms a conjecture of
	J.~Dubois, S.~Friedl, and W.~L\"uck and addresses a question of W.~Li and W.~Zhang. 
	Associated to any admissible homomorphism $\gamma:\pi_1(N)\to G$,
	the $L^2$--Alexander torsion $\tau^{(2)}(N,\gamma,\phi)$ is shown to be 
	continuous and everywhere positive provided that $G$ is residually finite
	and $(N,\gamma)$ is weakly acyclic. In this case, 
	a generalized degree can be assigned to $\tau^{(2)}(N,\gamma,\phi)$.
	Moreover, the generalized degree is	bounded by the Thurston norm of $\phi$.
\end{abstract}

\maketitle

\section{Introduction}
		Let $N$ be an irreducible orientable compact $3$-manifold with empty or incompressible toral boundary.
		Given a homomorphism $\gamma:\pi_1(N)\to G$ to a countable target group $G$
		and a cohomology class $\phi\in H^1(N;\,\Real)$,
		the triple $(\pi_1(N),\gamma,\phi)$ is said to be \emph{admissible} if 
		the homomorphism $\pi_1(N)\to \Real$ induced by $\phi$ factors through $\gamma$.
		 Associated to any given admissible triple,
		the \emph{$L^2$--Alexander torsion}
		has been introduced by J\'{e}r\^{o}me Dubois, Stefan Friedl, and Wolfgang L\"{u}ck 
		\cite{DFL-torsion}. It is a function		
			$$\tau^{(2)}(N,\gamma,\phi):\,\Real_+\to[0,+\infty),$$
		uniquely defined up to multiplication by a function of the form $t\mapsto t^r$ where $r\in\Real$.
		In this paper, we use a dotted equal symbol to mean 
		two functions being equal to each other
		up to such a monic power function factor.
		When $\gamma$ is taken to be $\mathrm{id}_{\pi_1(N)}:\pi_1(N)\to\pi_1(N)$,
		the corresponding function is called the \emph{full $L^2$--Alexander torsion} 
		with respect to $\phi$,	denoted by $\tau^{(2)}(N,\phi)(t)$.
		In \cite{DFL-torsion,DFL-symmetric}, the following properties
		about the full $L^2$--Alexander torsion are proved:		
		\begin{enumerate}
		\item For all $c\in\Real$,
		$$\tau^{(2)}(N,\,c\phi)(t)\,\doteq\,\tau^{(2)}(N,\phi)(t^c).$$
		\item 
		$$\tau^{(2)}(N,-\phi)(t)\,\doteq\,\tau^{(2)}(N,\phi)(t).$$
		\item For any fibered class $\phi\in H^1(N;\Integral)$,
		$$\tau^{(2)}(N,\phi)(t)\,\doteq\,\begin{cases}1&t\in(0,e^{-h(\phi)})\\t^{x_N(\phi)}&t\in(e^{h(\phi)},+\infty)\end{cases}$$
		where $h(\phi)$ denotes the entropy of the monodromy,	and $x_N(\phi)$ denotes the Thurston norm.
		\item Denoting by $\mathrm{Vol}(N)$ the simplicial volume of $N$,
		$$\tau^{(2)}(N,\phi)(1)\,=\,e^{\frac{\mathrm{Vol}(N)}{6\pi}}.$$
		\item If $\mathrm{Vol}(N)$ equals $0$,
		$$\tau^{(2)}(N,\phi)(t)\,\doteq\,\begin{cases}1&t\in(0,1]\\t^{x_N(\phi)}&t\in[1,+\infty)\end{cases}$$
		\end{enumerate}

	For knot complements, the full $L^2$--Alexander torsion recovers the $L^2$--Alexander invariant introduced earlier
	by Weiping Li and Weiping Zhang \cite{LZ-Alexander,LZ-AlexanderConway}. If $\gamma$ is virtually abelian,
	the $L^2$--Alexander torsion is closely related to the twisted Alexander polynomial through certain function
	associated to the Mahler measure \cite{DFL-torsion}. We refer the reader to the survey \cite{DFL-flavors}
	for more relations between the $L^2$--Alexander torsion and other flavors of Alexander-type invariants.
	
	It is generally anticipated that the degree of Alexander-type invariants conveys topological information about
	of the cohomology class $\phi$ of $N$. For example, the degree of twisted Alexander polynomials can be used to 
	detect the Thurston norm of $\phi$ due to Stefan Friedl and Stefano Vidussi \cite{FV-detectThurstonNorm}.
	Various comparison results are also known, cf.~\cite{Cochran,FK-norm,Harvey-degree,McMullen,Turaev,Vidussi-norm}.
	For $L^2$--Alexander torsion, a fundamental problem is to define the degree in the first place.
	The following version	has been proposed by Dubois--Friedl--L\"uck \cite[Section 1.2]{DFL-torsion}, 
	(there simply called the degree):	
	
	\begin{definition}\label{degree-a}
		Let $f:\Real_+\to [0,+\infty)$ be a function. Suppose that $f$ is asymptotically monomial in both ends,
		namely, as $t\to+\infty$, the following asymptotic formula
		holds for some constants $C_{+\infty}\in \Real_+$ and $d_{+\infty}\in\Real$:
			$$f\sim C_{+\infty}\cdot t^{d_{+\infty}},$$
		and the same property holds with $+\infty$ replaced by $0+$.
		Here the notation $f\sim g$ means that the ratio between the functions on both sides tends to $1$.
		For such $f$, the \emph{asymptote degree} of $f$ is defined to be the value:
			$$\mathrm{deg}^{\mathtt{a}}(f)\,=\,d_{+\infty}-d_{0+}\,\in\,\Real.$$
	\end{definition}
	
	The main goal of this paper is to establish the existence of 
	the asymptote degree for the full $L^2$--Alexander torsion of $3$-manifolds,
	and confirm in this case that the degree equals the Thurston norm:

	\begin{theorem}\label{main-torsion}
		Let $N$ be an irreducible orientable compact $3$-manifold with empty or incompressible toral boundary.
		Given any cohomology class $\phi\in H^1(N;\,\Real)$,
		the following properties hold true for any representative of the full $L^2$--Alexander torsion $\tau^{(2)}(N,\phi)$.
		\begin{enumerate}
		\item The function $\tau^{(2)}(N,\phi)(t)$ is continuous and everywhere positive,
			defined for all $t\in\Real_+$. In fact,
			the function $\tau^{(2)}(N,\phi)(t)\cdot\max\{1,t\}^m$
			is multiplicatively convex for any sufficiently large positive constant $m$,
			where the bound depends on $N$ and $\phi$.
		\item As the parameter $t$ tends to $+\infty$,
			$$\tau^{(2)}(N,\phi)(t)\,\sim\,C(N,\phi)\cdot t^{d_{+\infty}}$$
			for some constant $d_{+\infty}\in\Real$ and some constant
			$$C(N,\phi)\,\in\left[1,\,e^{\mathrm{Vol}(N)/6\pi}\right].$$
			The same asymptotic formula holds true for with $+\infty$ replaced by $0+$.
		\item	Hence the asymptote degree of $\tau^{(2)}(N,\phi)$ is valid. 
		Furthermore, 
			$$\mathrm{deg}^{\mathtt{a}}\left(\tau^{(2)}(N,\phi)\right)\,=\,x_N(\phi).$$
		\item The leading coefficient function
			\begin{eqnarray*}
				H^1(N;\Real)&\to& \left[1,\,e^{\mathrm{Vol}(N)/6\pi}\right]\\
				\phi&\mapsto & C(N,\phi)
			\end{eqnarray*}
			is upper semicontinuous.
		\end{enumerate}
	\end{theorem}
	
	In particular, Theorem \ref{main-torsion} confirms 
	Conjecture 1.1 (1) of Dubois--Friedl--L\"uck \cite{DFL-flavors}.
	In fact, many aspects of Theorem \ref{main-torsion}
	have also been conjectured, at least for knot complements, cf.~
	\cite[Subsection 5.8]{DFL-flavors}. In particular, the first part of
	Theorem \ref{main-torsion} addresses the question (Q2) of 
	Li--Zhang \cite{LZ-AlexanderConway}.
	
	The full $L^2$--Alexander torsion apparently loses information about 
	the fiberedness of cohomology classes in general.
	In fact, we have already observed that 
	the full $L^2$--Alexander torsion of graph manifolds 
	is completely determined by the Thurston norm, \cite[Theorem 1.2]{DFL-torsion}, \cite{Herrmann}.
	However, we exhibit an example at the end of this paper to indicate that nontrivial leading coefficients
	could occur, (Section \ref{Sec-example}). 
	The example might suggest that the leading coefficient $C(N,\phi)$ retains some information
	about the cohomology class $\phi$ which is volume (of the 3-dimensional hyperbolic type) in nature.
	For a primitive classes $\phi\in H^1(N;\Integral)$,
	we hence wonder if $C(N,\phi)$ measures certain volume of the guts 
	if one decomposes $N$ along a taut subsurface dual to $\phi$.
			
	It is possible to prove an analogous comparison theorem for more general $L^2$--Alexander torsions.
	To this end, we introduce another degree
	under less strict requirements:
	
	\begin{definition}\label{degree-b}
		Let $f:\Real_+\to [0,+\infty)$ be a function. Suppose that the following supremum and infimum
		exist in $\Real$:
			$$\mathrm{deg}^{\mathtt{b}}_{+\infty}(f)\,=\,\inf\left\{D_{+\infty}\in\Real\,:\,\lim_{t\to+\infty}f(t)\cdot t^{-D_{+\infty}}\,=\,0\right\},$$
		and
			$$\mathrm{deg}^{\mathtt{b}}_{0+}(f)\,=\,\sup\left\{D_{0+}\in\Real\,:\,\lim_{t\to0+}f(t)\cdot t^{-D_{0+}}\,=\,0\right\}.$$			
		For such $f$, the \emph{growth bound degree} of $f$ is defined to be the value:
			$$\mathrm{deg}^{\mathtt{b}}(f)\,=\,\mathrm{deg}^{\mathtt{b}}_{+\infty}-\mathrm{deg}^{\mathtt{b}}_{0+}\,\in\,\Real.$$
	\end{definition}
	
	By saying that a pair $(N,\gamma)$ is \emph{weakly acyclic}, we mean that
	there are no non-vanishing $L^2$--Betti numbers
	for the covering space of $N$ that corresponds to $\mathrm{Ker}(\gamma)$,
	regarded as an $\mathrm{Im}(\gamma)$--space,
	cf.~\cite[Section 6.5]{Lueck-book}.
				
	\begin{theorem}\label{main-torsion-weak}
		Let $N$ be an irreducible orientable compact $3$-manifold with empty or incompressible toral boundary,
		and $\gamma:\pi_1(N)\to G$ be a homomorphism.
		Suppose that $G$ is finitely generated and residually finite, and
		$(N,\gamma)$ is weakly acyclic.
		Then the following properties hold true for any representative of the $L^2$--Alexander torsion 
		$\tau^{(2)}(N,\gamma,\phi)$ of any admissible triple $(N,\gamma,\phi)$ over $\Real$.
		\begin{enumerate}
		\item The function $\tau^{(2)}(N,\gamma,\phi)(t)$ is continuous and everywhere positive,
			defined for all $t\in\Real_+$. In fact,
			the function $\tau^{(2)}(N,\gamma,\phi)(t)\cdot\max\{1,t\}^m$
			is multiplicatively convex for any sufficiently large positive constant $m$,
			where the bound depends on $(N,\gamma,\phi)$.
		\item The growth bound degree of $\tau^{(2)}(N,\gamma,\phi)$ is valid.
		Furthermore, 
			$$\mathrm{deg}^{\mathtt{b}}\left(\tau^{(2)}(N,\gamma,\phi)\right)\,\leq\,x_N(\phi).$$
		\item The degree function
			\begin{eqnarray*}
				H^1(G;\Real)&\to& \Real\\
				\xi&\mapsto & \mathrm{deg}^{\mathtt{b}}\left(\tau^{(2)}(N,\gamma,\phi+\gamma^*\xi)\right)
			\end{eqnarray*}
			is Lipschitz continuous.
		\end{enumerate}
	\end{theorem}
	
	In a weaker form, Theorem \ref{main-torsion-weak}	generalizes 
	the virtually abelian case which has been done in \cite{DFL-torsion}.
	For example, if $N$ is a compact orientable surface bundle over the circle and $\gamma$ is a homomorphism
	of $\pi_1(N)$	onto a residually finite group $G$ such that $\gamma^*:H^1(G;\Real)\to H^1(N;\Real)$
	is onto, then the assumptions of Theorem \ref{main-torsion-weak} are satisfied.
	
	\begin{remark}
		Completely independently from work of this paper, Friedl and L\"uck 
		have also proved the equality between the (growth bound) degree of the full $L^2$-Alexander torsion 
		and the Thurston norm \cite{FriedlLueck-new}.
		In fact, their work implies Theorem \ref{main-torsion-weak} (2) as well. 
		Moreover,	their work relies on a systematic study of twisting $L^2$-invariants by L\"uck \cite{Lueck-new}.
		We point out that both \cite{Lueck-new} and \cite{FriedlLueck-new} keep 
		track of the Euler structures more closely than this paper does,
		which should be important for potential applications.
		For example, with a fixed Euler structure,
		the $L^2$--Alexander torsion becomes a genuine function in the pair $(\phi,t)$,
		so it would make sense to study its continuity and other properties.
	\end{remark}
	
	In the rest of the introduction, we discuss some ingredients involved 
	in the proof of Theorem \ref{main-torsion}. Theorem \ref{main-torsion-weak}
	can be proved along the way.
	After choosing some CW complex structure of $N$ convenient for calculation
	as used in \cite{DFL-torsion},
	we may manipulate $\tau(N,\phi)(t)$ into an alternating product, where the factors
	are regular Fuglede--Kadison determinants of the $L^2$--Alexander twist of 
	square matrices over $\Integral\pi_1(N)$. Except the one
	coming from the boundary homomorphism between dimension 2 and dimension 1,
	the factors are all very simple and well understood.
	Therefore, the proof of Theorem \ref{main-torsion} can be reduced to the study
	of the regular Fuglede--Kadison determinant for an $L^2$--Alexander twist of 
	a single matrix $A$.
	Associated to the admissible triple $(\pi_1(N),\mathrm{id}_{\pi_1(N)},\phi)$, 
	the factor corresponding to $A$ is a non-negative function 
	defined for $t\in\Real_+$ of the form
		$$V(t)\,=\,\mathrm{det}^{\mathtt{r}}_{\mathcal{N}(G)}\left(\kappa(\phi,\mathrm{id}_{\pi_1(N)},t)(A)\right),$$
	where $A$ is a square matrix over $\Integral\pi_1(N)$,
	(cf.~Section \ref{Sec-prelim} for the notations).
	
	The first ingredient is to show that $V(t)$ is a multiplicatively convex function with bounded exponent.
	See Section \ref{Sec-mConvexFunction} for the terminology.
	In fact, we show in Theorem \ref{mConvex-eBounded} that the asserted property 
	holds true for general admissible triples $(\pi,\gamma,\phi)$ over $\Real$ and square matrices
	$A$ over $\Complex\pi$,
	as long as the target group $G$ of $\gamma$ is residually finite.
	The exponent bound can be easily perceived, and can be easily proved once the multiplicative
	convexity is available.
	When $G$ is finitely generated and virtually abelian, the multiplicative convexity can be verified by
	computation using Mahler measure of multivariable Laurent polynomials.
	Therefore, to approach the residually finite case, it is natural to consider 
	a cofinal tower of virtually abelian quotients $G$, denoted as
		$$G\to\cdots\to\Gamma_n\to\cdots\to\Gamma_2\to\Gamma_1,$$
	which gives rise to a sequence of $L^2$--Alexander twist homomorphisms $\kappa(\gamma_n,\phi,t)$,
	where $\gamma_n:\pi\to\Gamma_n$ is the induced homomorphism.
	For any given $t\in\Real_+$,
	the spectra of the matrices $A_n(t)=\kappa(\gamma_n,\phi,t)(A)$ could become increasingly
	dense near $0$, as $n$ tends to $\infty$,
	so it should not be expected in general that the sequence of functions
	$\mathrm{det}^{\mathtt{r}}_{\mathcal{N}(\Gamma_n)}(A_n(t))$
	converged pointwise to $V_G(t)=\mathrm{det}^{\mathtt{r}}_{\mathcal{N}(G)}(A(t))$.
	By introducing a positive $\epsilon$-pertubation of 
	the positive operator $A_n(t)^*A_n(t)$, namely,
	$$H_{n,\epsilon}(t)\,=\,A_n(t)^*A_n(t)+\epsilon\cdot\mathbf{1},$$
	the issue of small spectrum values can be bypassed.
	However, one has to be careful because of the fact that the $L^2$--Alexander twist does not commute with the operation of
	taking self-adjoint. For example, $H_{n,\epsilon}(t)$ is in general not a family of $L^2$--Alexander twisted operators,
	so the regular determinant of $H_{n,\epsilon}(t)$
	does not need to be multiplicatively convex in the parameter $t\in\Real_+$.
	Instead of arguing that way, for any fixed $T\in\Real_+$, we look at the functions
	$$W_{n,\epsilon}(s,T)=\mathrm{det}^{\mathtt{r}}_{\mathcal{N}(\Gamma_n)}\left(\kappa(\gamma_{n*}\phi,\mathrm{id}_{\Gamma_n},s)(H_{n,\epsilon}(T))\right)$$
	in a new parameter $s\in\Real_+$.
	As $n\to\infty$ and then $\epsilon\to0+$, we show that
	$W_{n,\epsilon}(1,T)$ converges to $W_{\infty,0}(1,T)$, while the limit superior of $W_{n,\epsilon}(s,T)$ 
	does not exceed $W_{\infty,0}(s,T)$. Using the fact that $W_{n,\epsilon}(s,T)$ are multiplicatively convex in $s\in\Real_+$,
	it can be implied that $V_G(t)$ is multiplicatively convex as well.
		
	The growth bound degree is applicable to any (nowhere zero) multiplicatively convex
	function with bounded exponent.
	It can be equivalently characterized as the width of the range of
	all possible exponents (or `multiplicative slopes') 
	between pairs of points.
	As a consequence of Theorem \ref{mConvex-eBounded},
	we are able to show that the growth bound degree $\mathrm{deg}^{\mathtt{b}}(V)$
	depends Lipschitz-continuously on the cohomology class $\phi\in H^1(N;\Real)$, (Theorem \ref{continuityOfDegree}).
	
	The second ingredient is a criterion to confirm
	that $V(t)$ is asymptotically monomial as $t$ tends to $+\infty$ or $0+$,
	or in other words, that $\mathrm{deg}^{\mathtt{a}}(V)$ equals $\mathrm{deg}^{\mathtt{b}}(V)$.
	To motivate the conditions, consider the sequence of determinant functions 
	$$V_n(t)\,=\,\mathrm{det}^{\mathtt{r}}_{\mathcal{N}(\Gamma_n)}(\kappa(\gamma_n,\phi,t)(A))$$
	associated to the cofinal tower of virtually abelian quotients $\Gamma_n$ above.
	Using techniques of \cite{Lueck-approximating}, what one can show is that for every $t\in\Real_+$,  as $n\to\infty$,
	the supremum limit of $V_n(t)$ does not exceed $V(t)$. On the other hand, the functions $V_n(t)$ 
	are all multiplicatively convex
	and asymptotically monomial in both ends. As $t\to+\infty$, suppose
		$$V_n(t)\sim C_{+\infty,n}\cdot t^{d_{+\infty,n}},$$
	and similarly we introduce the notations $C_{0+,n}$ and $d_{0+,n}$ for $t\to0+$.
	As $n\to\infty$, if the degrees 
	$\mathrm{deg}^{\mathrm{b}}(V_n)=\mathrm{deg}^{\mathrm{a}}(V_n)=d_{+\infty,n}-d_{0+,n}$ converge 
	to the growth bound degree $\mathrm{deg}^{\mathrm{b}}(V)$, 
	and if the coefficients $C_{+\infty,n}$ and $C_{0+,n}$ are uniformly bounded
	below by some constant $L\in\Real_+$,
	then it can be implied by the geometry of the log--log plots of the functions
	that $V(t)$ must be asymptotically monomial in both ends as well, (Lemma \ref{mConvexVersion}).
	
	For our proof of Theorem \ref{main-torsion}, the convergence of growth bound degrees
	can be guaranteed by the virtual RFRS property of $3$-manifold groups, at least
	after excluding the case of graph manifolds, which has been treated by \cite[Theorem 1.2]{DFL-torsion}, \cite{Herrmann}.
	In fact, combined with the continuity of degree that we have already mentioned,
	the method of \cite[Theorem 9.1]{DFL-torsion} can be applied to produce a cofinal tower
	of virtually abelian quotients
	such that the growth bound degree of each $V_n(t)$ and $V(t)$ is equal to the Thurston norm of $\phi$.
	On the other hand, based on the fact that 
	$A$ is a square matrix over $\Integral\pi_1(N)$,
	computation shows that the coefficients $C_{+\infty,n}$ and
	$C_{0+,n}$ are all radicals of the Mahler measure of certain multivariable
	Laurent polynomial over $\Integral$. 
	This yields a uniform lower bound $1$ for all the coefficients.
	Therefore, the criterion is applicable to our situation,
	and we can complete the proof of Theorem \ref{main-torsion}.
	
	In Section \ref{Sec-prelim}, we recall some terminology that is used in this paper. 
	In Section \ref{Sec-rFKdet}, we introduce regular Fuglede--Kadison determinants and
	discuss its limiting behavior. In Section \ref{Sec-mConvexFunction}, 
	we introduce multiplicatively convex functions
	and mention some basic properties. 
	After these preparing sections,
	we study the regular Fuglede--Kadison determinants of matrices under
	$L^2$--Alexander twists in Sections \ref{Sec-mConvex-eBounded}, \ref{Sec-continuityOfDegree}, and \ref{Sec-asymptotics}:
	The multiplicative convexity and the existence of the growth bound degree is shown in Section \ref{Sec-mConvex-eBounded};
	The continuity of degree is derived in Section \ref{Sec-continuityOfDegree}; The criterion for monomial asymptotics
	is introduced in Section \ref{Sec-asymptotics}.
	In Section \ref{Sec-mainProofs}, we apply the ingredients to $L^2$--Alexander torsions of $3$-manifolds,
	and prove Theorems \ref{main-torsion} and \ref{main-torsion-weak}.
	In Section \ref{Sec-example}, we give an example regarding nontrivial leading coefficients.

	\subsection*{Acknowledgements} The author would like to thank Stefan Friedl and Wolfgang L\"{u}ck for letting him learn their
	independent work and for subsequent valuable communications. 
	The author also thanks Weiping Li for interesting conversations.

\section{Preliminaries}\label{Sec-prelim}
	In this section, we recall some terminology of Dubois--Friedl--L\"uck \cite{DFL-torsion}.
	We also briefly recall some fundamental facts in 3-manifold topology. 
	For background in $L^2$-invariants, including group von Neumann algebras and 
	Fuglede--Kadison determinants, we refer the reader to the book of W.~L\"uck \cite{Lueck-book}.
	
	\subsection{Admissible triples}
	Admissibility conditions have been introduced by S.~Harvey 
	for study of higher-order Alexander polynomials \cite[Definition 1.4]{Harvey-monotonicity}.
	In this paper, we adopt the following notations, according to \cite{DFL-torsion}.

	\begin{definition}
	Let $L\subset \Real$ be any additive group of real numbers, for example, $\Integral$, $\Rational$, or $\Real$. 
	Given a countable group $\pi$, and a homomorphism $\phi\in \mathrm{Hom}(\pi,L)$, and
	a homomorphism $\gamma:\pi\to G$ to any countable group $G$, 
	we say that $(\pi,\phi,\gamma)$ forms \emph{an admissible triple
	over $L$}	if $\phi$ factors through $\gamma$.
	That is, for some homomorphism $G\to L$, the following diagram commutes:	
	$$\xymatrix{
	\pi \ar[r]^\gamma \ar[rd]_\phi &G \ar[d] \\& L}$$
	\end{definition}
	
	Given any positive real parameter $t\in \Real_+$, there is a homomorphism of rings:
		$$\kappa(\phi,\gamma,t):\,\Integral\pi\longrightarrow \Real G$$
	defined uniquely by
		$$\kappa(\phi,\gamma,t)(g)\,=\,t^{\phi(g)}\gamma(g)$$
	for all $g\in\pi$ via linear extension over $\Integral$.
	Then for any positive integer $p$, $\kappa(\phi,\gamma,t)$ naturally extends to be a homomorphism of algebras:
		$$\kappa(\phi,\gamma,t):\,\Matrix_{p\times p}(\Complex\pi)\to \Matrix_{p\times p}(\Complex G)$$
	by applying $\kappa(\phi,\gamma,t)$ to entries accordingly.
	
	Note that $\kappa(\phi,\gamma,t)$ is not a homomorphism of $*$-algebras in general. 	
	In fact,
	$$\kappa(\phi,\gamma,t)(A)^*=\kappa(\phi,\gamma,t^{-1})(A^*).$$
	Recall that for any square matrix $A=(a_{ij})_{p\times p}$ over
	$\Complex G$, as an operator of $\ell^{2}(G)^{\oplus p}$,
	the adjoint operator can be given by $A^*=(a^*_{ji})_{p\times p}$,
	where the involution of an element $a=\sum_{k} a_kg_k\in\Complex G$ is given by
	$a^*=\sum_k \bar{a}_kg_k^{-1}\in\Complex G$.

	Every admissible triple $(\pi,\phi,\gamma)$ over $L$ sits naturally in an affine family of 
	admissible triples parametrized by $\mathrm{Hom}(G,L)$. 
	Specifically, for any homomorphism
		$$\xi\in\mathrm{Hom}(G,L),$$
	we have a new admissible triple $(\pi,\phi+\gamma^*\xi,\gamma)$,
	where $\phi+\gamma^*\xi:\pi\to L$ is the homomorphism defined by
		$$(\phi+\gamma^*\xi)(g)\,=\,\phi(g)+\xi(\gamma(g))$$
	for all $g\in\pi$.
	To speak of continuity, we consider the space $\mathrm{Hom}(G,L)$ to be equipped with
	the compact-open topology, regarding $G$ to be a discrete group 
	and $L$ have the subspace topology of $\Real$.
	
	\begin{lemma}\label{homologicallyIsomorphic}
		If $\gamma:\pi\to G$ induces an isomorphism $\gamma_*:H_1(\pi;\,\Real)\to\,H_1(G;\,\Real)$,
		then $(\pi,\gamma,\phi)$ is admissible for every homomorphism $\phi:\pi\to\Real$.
	\end{lemma}
	
	\begin{proof}
		In this case, the composition 
			$$\pi\stackrel{\gamma}\longrightarrow G\longrightarrow H_1(G;\,\Real) \stackrel{\gamma_*^{-1}}\longrightarrow 
			H_1(\pi;\,\Real)\stackrel{\phi_*}\longrightarrow \Real$$
		recovers the homomorphism $\phi$.
	\end{proof}

	\subsection{$L^2$--Alexander torsion}
	Let $X$ be a connected finite CW complex. The universal cover $\widehat{X}$ of $X$ is a CW complex equipped with a free
	action of the deck transformation group $\pi_1(X)$. We equip the chain complex $C_*(\widehat{X})$ 
	with a left $\Integral\pi_1(X)$	action induced by the deck transformation. On the other hand, given any 
	admissible triple $(\pi_1(X),\gamma,\phi)$ over $\Real$, and given a parameter value $t\in\Real_+$,
	we may equip the Hilbert space $\ell^2(G)$ with a right $\Integral\pi_1(X)$--module structure via the representation:
		$$\kappa(\phi,\gamma,t):\,\Integral\pi_1(X)\longrightarrow \Real G.$$
	In this paper, we treat $\ell^2(G)$ as a right $\Real G$--module 
	and a left Hilbert $\mathcal{N}(G)$--module.
	Here we denote by
		$$\mathcal{N}(G)\,=\,\mathcal{B}\left(\ell^2(G)\right)^{G}$$
	the group von Neumann algebra of $G$ which 
	consists of all the bounded operators that commutes with the right multiplication by elements of $G$.		
	Twisting the chain complex of $\widehat{X}$ by the module $\ell^2(G)$ via the representation $\kappa(\phi,\gamma,t)$
	gives rise to a (left) Hilbert $\mathcal{N}(G)$--chain complex:
		$$\ell^2(G)\otimes_{\Integral\pi_1(X)}C_*(\widehat{X})$$
	and the twisted boundary homomorphism is defined by $\mathbf{1}\otimes \partial_*$.
	In fact, the twisted complex is finitely generated and free over $\mathcal{N}(G)$.
	In other words, by choosing a lift of each cell of $X$ in $\widehat{X}$, each chain module
	of the complex can be identified with a direct sum of the regular Hilbert $\mathcal{N}(G)$-modules:
	$$\ell^2(G)\otimes_{\Integral\pi_1(X)}C_k(\widehat{X})\,\cong\,\ell^2(G)^{\oplus p_k}.$$
	In this paper, we restrict ourselves to finitely generated, free
	Hilbert $\mathcal{N}(G)$--chain complexes which
	are \emph{weakly acyclic and of determinant class}.
	This means that
	the $\ell^2$-Betti numbers are all trivial and all the Fuglede--Kadison determinants of 
	the boundary homomorphisms take values in $(0,+\infty)$.
	In such case, the \emph{$L^2$--Alexander torsion} of $X$ at $t$ with respect to $\gamma$ and $\phi$
	is defined to be the multiplicatively alternating product of the Fuglede--Kadison determinants of the boundary
	homomorphisms:
		$$\tau^{(2)}(X,\gamma,\phi)(t)\,\doteq\,\prod_{k\in\Integral} \mathrm{det}_{\mathcal{N}(G)}(\mathbf{1}\otimes\partial_k)^{(-1)^k}.$$
	Here the dotted equal means that we treat the $L^2$--Alexander torsion as
	a function in the parameter $t\in\Real_+$. In fact, choosing another collection of lifts
	may result in a change of the value of the right-hand side by a multiplicative factor
	$t^r$, for some exponent $r\in\Real$ independent of $t$, 
	so the function $\tau^{(2)}(X,\gamma,\phi)$
	is well defined only up to a monic power function factor.
	We remark that our notational convention follows \cite{DFL-torsion},
	and the exponential of the $L^2$-torsion
	according to \cite[Definition 3.29]{Lueck-book} is the multiplicative inverse of the $\tau^{(2)}$ above.
	To be convenient, a value $0$ is artificially assigned to $\tau^{(2)}(X,\gamma,\phi)(t)$
	if the twisted complex fails to be weakly acyclic or of determinant class.
	With this convention, the $L^2$--Alexander torsion associated to $(X,\gamma,\phi)$ 
	is a function determined up to a monic power function factor:
		$$\tau^{(2)}(X,\gamma,\phi):\,\Real_+\to[0,+\infty).$$
	
	Let $N$ be a compact smooth manifold, possibly with boundary,
	and $\gamma:\pi_1(N)\to G$ be a homomorphism. 
	The \emph{$L^2$--Alexander torsion} of $N$
	with respect to any admissible triple	$(\pi_1(N),\gamma,\phi)$, denoted as
	$\tau^{(2)}(N,\gamma,\phi)$, is understood 
	to be the $L^2$--Alexander torsion of any finite CW complex structure of $N$. 
	This notion does not depend on 
	the choice	of the CW structure \cite[Section 4.2]{DFL-torsion}.
	When $\gamma$ is taken to be $\mathrm{id}_{\pi_1(N)}:\pi_1(N)\to\pi_1(N)$,
	the triple $(\pi_1(N),\gamma,\phi)$ is admissible for every class $\phi\in H^1(N;\Real)$.
	The corresponding $L^2$--Alexander torsion is called the \emph{full $L^2$--Alexander torsion}
	with respect to $\phi$, denoted as $\tau^{(2)}(N,\phi)$.

	\subsection{Thurston norm and virtual fibering}
	Let $N$ be an irreducible compact orientable $3$-manifold with empty or incompressible toral boundary.
	The \emph{Thurston norm}, named after William P.~Thurston who discovered it in \cite{Thurston-norm}, 
	is a seminorm of the vector space:
	$$x_N:\,H_2(N,\partial N;\,\Real)\to[0,+\infty),$$
	which takes $\Integral$ values on the integral lattice $H_2(N,\partial N;\,\Integral)$.
	The Thurston norm measures certain complexity of the second relative homology classes,
	and it is known to be non-degenerate if the $3$-manifold $N$ supports a complete hyperbolic
	structure in its interior.
	The unit ball $B_x(N)$ of $x_N$ is a convex polyhedron, symmetric about the origin,
	and supported by finitely many linear faces	carried by rational affine hyperplanes.
	If $N$ fibers over the circle via a map $N\to S^1$, any fiber of the fibration
	represents a homology class  $[\Sigma]\in H_2(N,\partial N;\Integral)$, which depends only
	on the fibration.	We can canonically identify $[\Sigma]$ with
	a cohomology class $\phi\in H^1(N;\Integral)\cong[N,S^1]$, 
	by Poincar\'e Duality (after fixing an orientation of $N$).
	As we have assumed $N$ to be irreducible with incompressible boundary, $x_N(\phi)$ equals
	$-\chi(\Sigma)$.
	Any such $\phi$ is called a \emph{fibered class}.
	Thurston has shown that every fibered class is contained in the open cone over a top-dimensional face 
	of $\partial B_x(N)$, and every integral class of that cone is a fibered class.
	Such open cones are hence called the \emph{fibered cones} of $x_N$. 
	
	In general, $N$ may possess no fibered cones at all. However, given any class $\phi\in H^1(N;\Real)$, 
	we can usually pass to a finite cover $p:\,\tilde{N}\to N$,
	so that $p^*\phi\in H^1(\tilde{N};\Real)$ is \emph{quasi-fibered},
	namely, $p^*\phi$ lies on the (point-set theoretic) boundary of a fibered cone
	possessed by $\tilde{N}$. To be precise, the virtual quasi-fibering property
	holds true for every class $\phi\in H^1(N;\Real)$
	if $\pi_1(N)$ is virtually residually finite rationally solvable (or RFRS),
	due to a theorem of Ian Agol \cite{Agol-RFRS}.
	Based on the confirmations of the
	Virtual Haken Conjecture and the Virtual Fibering Conjecture due to the works
	of Ian Agol \cite{Agol-VHC}, Daniel Wise \cite{Wise-book}, and many other authors,
	it has been known that $\pi_1(N)$ is virtually RFRS if and only if
	$N$ supports a complete Riemannian metric of nonpositive curvature in its interior,
	\cite{Liu,PW-graph,PW-mixed}.
	For example, if the simplicial volume $\mathrm{Vol}(N)$ is positive, or in other words,
	if $N$ contains at least one hyperbolic piece in its geometric decomposition,
	then the virtual quasi-fibering property is possessed by $N$.
	We refer the reader to the survey \cite{AFW-group} for more background about virtual properties
	of $3$-manifolds.

\section{Regular Fuglede--Kadison determinant}\label{Sec-rFKdet}
	Let $G$ be a countable discrete group. For any $p\times p$ matrix $A$ over $\mathcal{N}(G)$, 
	the \emph{regular Fuglede--Kadison determinant} of $A$ is defined to be
	$$\mathrm{det}^{\mathtt{r}}_{\mathcal{N}(G)}(A)
	=\begin{cases} \mathrm{det}_{\mathcal{N}(G)}(A)&\textrm{if }A\textrm{ is full rank of determinant class}
	\\0&\textrm{otherwise}\end{cases}$$
	This gives rise to a function:
	$$\mathrm{det}^{\mathtt{r}}_{\mathcal{N}(G)}:\,\Matrix_{p\times p}(\mathcal{N}(G))\to [0,+\infty).$$
	
	Regular Fuglede--Kadison determinants have been used in \cite{DFL-torsion}. In the rest of the section
	we study the semicontinuity of this quantity under two kinds of limiting
	processes.
	
	\begin{lemma}\label{norm-semicontinuous}
		If a sequence of $p\times p$ matrices $\{A_n\}_{n\in\Natural}$ over $\mathcal{N}(G)$
		converges to $A\in\Matrix_{p\times p}(\mathcal{N}(G))$ with respect to the norm topology,
		then
		$$\limsup_{n\to\infty} \mathrm{det}^{\mathtt{r}}_{\mathcal{N}(G)}(A_n)\,\leq\,\mathrm{det}^{\mathtt{r}}_{\mathcal{N}(G)}(A).$$
		Moreover, if $A$ is a positive operator, then
		$$\lim_{\epsilon\to0+}\mathrm{det}^{\mathtt{r}}_{\mathcal{N}(G)}(A+\epsilon\cdot\mathbf{1})\,=\,
		\mathrm{det}^{\mathtt{r}}_{\mathcal{N}(G)}(A).$$
	\end{lemma}
	
	\begin{proof}
		Since $\mathrm{det}^{\mathtt{r}}_{\mathcal{N}(G)}(A^*A)$ equals $\mathrm{det}^{\mathtt{r}}_{\mathcal{N}(G)}(A)^2$,
		it suffices to show the inequality for positive operators $\{A_n\}_{n\in\Natural}$ and $A$. For any arbitrary
		constant $\epsilon>0$, the positive operators $(A_n+\epsilon\cdot\mathbf{1})$ and $(A+\epsilon\cdot\mathbf{1})$
		are invertible, so the regular Fuglede--Kadison determinant agrees with the Fuglede--Kadison determinant.
		Since the Fuglede--Kadison determinant is continuous on the subgroup of invertible matrices
		$\mathrm{GL}(p,\mathcal{N}(G))$ with respect to the norm topology \cite[Theorem 1.10 (d)]{CFM},
		$$\lim_{n\to\infty} \mathrm{det}^{\mathtt{r}}_{\mathcal{N}(G)}(A_n+\epsilon\cdot\mathbf{1})
		\,=\,\mathrm{det}^{\mathtt{r}}_{\mathcal{N}(G)}(A+\epsilon\cdot\mathbf{1}).$$
		On the other hand, by \cite[Lemma 3.15 (6)]{Lueck-book}, or as a trivial fact if $A_n$ fails to be injective,
		$$\mathrm{det}^{\mathtt{r}}_{\mathcal{N}(G)}(A_n)\leq \mathrm{det}^{\mathtt{r}}_{\mathcal{N}(G)}(A_n+\epsilon\cdot\mathbf{1}).$$
		Therefore,
		\begin{eqnarray*}
			\limsup_{n\to\infty} \mathrm{det}^{\mathtt{r}}_{\mathcal{N}(G)}(A_n)
			&\leq&\limsup_{n\to\infty} \mathrm{det}^{\mathtt{r}}_{\mathcal{N}(G)}(A_n+\epsilon\cdot\mathbf{1})\\
			&=&\mathrm{det}^{\mathtt{r}}_{\mathcal{N}(G)}(A+\epsilon\cdot\mathbf{1}).
		\end{eqnarray*}
		As $\epsilon>0$ is arbitrary, it suffices to prove
		$$\lim_{\epsilon\to0+}\mathrm{det}^{\mathtt{r}}_{\mathcal{N}(G)}(A+\epsilon\cdot\mathbf{1})\,=\,
		\mathrm{det}^{\mathtt{r}}_{\mathcal{N}(G)}(A).$$
		In fact, if $A$ is injective, the last limit follows from \cite[Lemma 3.15 (5)]{Lueck-book}.
		Otherwise, $\mathrm{det}^{\mathtt{r}}_{\mathcal{N}(G)}(A)$ equals $0$. Denoting by $b\in (0,p]$ 
		the von Neumann dimension $\dim_{\mathcal{N}(G)}\mathrm{Ker}(A)$, it is easy to estimate
			$$0\leq \mathrm{det}^{\mathtt{r}}_{\mathcal{N}(G)}(A+\epsilon\cdot\mathbf{1})\leq \epsilon^b(\|A\|+\epsilon)^{p-b}.$$
		We again have:
		$$\lim_{\epsilon\to0+}\mathrm{det}^{\mathtt{r}}_{\mathcal{N}(G)}
			(A+\epsilon\cdot\mathbf{1})\,=\,0\,=\,\mathrm{det}^{\mathtt{r}}_{\mathcal{N}(G)}(A).$$
		This completes the proof.		
	\end{proof}
			
	\begin{lemma}\label{stable-semicontinuous}
		Let
			$$G\to\cdots\to \Gamma_n\to\cdots\to \Gamma_2\to\Gamma_1,$$
		be a cofinal tower of quotients of $G$, and denote by $\psi_n:G\to\Gamma_n$
		the quotient homomorphisms.	
		Suppose that all the target groups $\Gamma_n$ are finitely generated and residually finite.
		Let $A_G$ be a square matrix over $\Complex G$.
		Then
				$$\limsup_{n\to\infty}\mathrm{det}^{\mathtt{r}}_{\mathcal{N}(\Gamma_n)}(\psi_{n*}A_G)\leq\mathrm{det}^{\mathtt{r}}_{\mathcal{N}(G)}(A_G).$$
		Moreover, for any constant $\epsilon>0$,
				$$\lim_{n\to\infty}\mathrm{det}^{\mathtt{r}}_{\mathcal{N}(\Gamma_n)}(\psi_{n*}(A_G^*A_G+\epsilon\cdot\mathbf{1}))
				=\mathrm{det}^{\mathtt{r}}_{\mathcal{N}(G)}(A_G^*A_G+\epsilon\cdot\mathbf{1}).$$
	\end{lemma}
	
	Here the tower being cofinal means that
			$$\bigcap_{n\in\Natural} \mathrm{Ker}{\psi_n}\,=\,\{\,\mathrm{id}_G\,\}.$$
	
	\begin{proof}
		Assuming that the `moreover' part has been proved, we can derive the first inequality as follows.
		For any constant $\epsilon>0$,
		\begin{eqnarray*}
			\mathrm{det}^{\mathtt{r}}_{\mathcal{N}(\Gamma_n)}(\psi_{n*}A_G)
			&=&\mathrm{det}^{\mathtt{r}}_{\mathcal{N}(\Gamma_n)}(\psi_{n*}(A^*_GA_G))^{1/2}\\
			&\leq&\mathrm{det}^{\mathtt{r}}_{\mathcal{N}(\Gamma_n)}(\psi_{n*}(A^*_GA_G+\epsilon\cdot\mathbf{1}))^{1/2}.
		\end{eqnarray*}
		The last expression tends to the regular Fuglede--Kadison determinant of $A_G$ as $\epsilon$ tends to $0+$, by Lemma \ref{norm-semicontinuous}.
		This implies the asserted inequality
		$$\limsup_{n\to\infty}\mathrm{det}^{\mathtt{r}}_{\mathcal{N}(\Gamma_n)}(\psi_{n*}A_G)\leq\mathrm{det}^{\mathtt{r}}_{\mathcal{N}(G)}(A_G).$$
		
		It remains to prove the asserted limit in the `moreover' part.
		For simplicity, given any constant $\epsilon>0$, we rewrite the matrices as
			$$H_\infty\,=\,A_G^*A_G+\epsilon\cdot\mathbf{1}\in\Matrix_{p\times p}(\Complex G)$$
		and
			$$H_n\,=\,\psi_{n*} H_\infty\in\Matrix_{p\times p}(\Complex\Gamma_n).$$			
		Note that the self-adjoint operators $H_n$ acting on $\ell^2(\Gamma_n)^{\oplus p}$ are positive with spectra bounded 
		uniformly	$\epsilon$ from $0$ and the same holds for $H_\infty$. In this case,
		approximation of determinants should follow from well known techniques.
		In the rest of the proof, we derive the approximation
			$$\lim_{n\to\infty}\mathrm{det}^{\mathtt{r}}_{\mathcal{N}(\Gamma_n)}(H_n)
				=\mathrm{det}^{\mathtt{r}}_{\mathcal{N}(G)}(H_\infty).$$
		from a theorem of W.~L\"uck \cite[Theorem 3.4 (3)]{Lueck-approximating},
		which is originally done for cofinal towers of finite quotients.
		
		It would be convenient to argue by contradiction, assuming that the limit of the left-hand side did not exist
		or did not equal to the right-hand side. In either case, 
		possibly after passing to a subsequence, we assume that
		there exists a constant $\delta>0$ such that the following gap estimate holds for all $n\in\Natural$:
		$$\left|\mathrm{det}^{\mathtt{r}}_{\mathcal{N}(\Gamma_n)}(H_n)-\mathrm{det}^{\mathtt{r}}_{\mathcal{N}(G)}(H_\infty)\right|\,\geq\,2\delta.$$
		
		By induction, we show that there exists a cofinal tower of finite quotients of $G$
			$$G\to\cdots\to \Gamma'_n\to\cdots\to \Gamma'_2\to\Gamma'_1,$$
		with the following properties: For all $n\in\Natural$, we have that $\Gamma'_n$ is a further quotient of $\Gamma_n$, 
		and moreover,
			$$\left|\mathrm{det}^{\mathtt{r}}_{\mathcal{N}(\Gamma'_n)}\left(H'_{n}\right)-
		\mathrm{det}^{\mathtt{r}}_{\mathcal{N}(\Gamma_n)}\left(H_n\right)\right|\,<\,\delta,$$
		where $H'_n$ is the induced matrix of $H_n$
		over $\Complex\Gamma'_n$. For $n$ equal to $1$,
		take a cofinal tower of finite quotients of $\Gamma_1$:
			$$\Gamma_1\to\cdots\to\Gamma_{1,j}\to\cdots\to\Gamma_{1,2}\to\Gamma_{1,1}.$$
		Denote the induced matrix of $H_n$ over $\Complex\Gamma_{1,j}$ by 
		$H_{1,j}$. Since $H_{1,j}$ is positive with spectrum bounded at least $\epsilon$ from $0$,
		L\"uck's theorem implies 
			$$\lim_{j\to\infty} \mathrm{det}^{\mathtt{r}}_{\mathcal{N}(\Gamma_{1,j})}\left(H_{1,j}\right)
			\,=\,\mathrm{det}^{\mathtt{r}}_{\mathcal{N}(\Gamma_1)}\left(H_{1}\right),$$
		so we choose $\Gamma'_1$ to be the quotient $\Gamma_{1,j}$ for a sufficiently large $j$.
		Suppose by induction that $\Gamma'_n$ has been constructed for some $n\in\Natural$.
		To construct $\Gamma'_{n+1}$, we take a tower of finite quotients
			$$\Gamma_{n+1}\to\cdots\to\Gamma_{n+1,j}\to\cdots\to\Gamma_{n+1,2}\to\Gamma_{n+1,1}.$$
		in the same fashion as above, but also require the first term $\Gamma_{n+1,1}$ to be $\Gamma'_n$.
		The same construction thus yields some sufficiently large $j$ such that
		$\Gamma_{n+1,j}$ can be chosen as $\Gamma'_{n+1}$. 
		This completes the induction.
		
		Provided with the new tower,
		L\"uck's theorem again implies
		$$\lim_{n\to\infty} \mathrm{det}^{\mathtt{r}}_{\mathcal{N}(\Gamma'_n)}\left(H'_n\right)\,=\,
			\mathrm{det}^{\mathtt{r}}_{\mathcal{N}(\Gamma_\infty)}\left(H_{\infty}\right).$$
		Therefore, for sufficiently large $n$,
		$$\left|\mathrm{det}^{\mathtt{r}}_{\mathcal{N}(\Gamma_n)}\left(H_n\right)
		-\mathrm{det}^{\mathtt{r}}_{\mathcal{N}(\Gamma_\infty)}\left(H_{\infty}\right)\right|\,<\,2\delta.$$
		This contradicts the assumed gap estimation, and hence completes the proof.		
	\end{proof}

\section{Multiplicatively convex function}\label{Sec-mConvexFunction}
	In this section, we give an introduction to multiplicatively convex functions.
	In subsequent sections, such functions arise naturally as we take the regular Fuglede--Kadison
	determinants of matrices under $L^2$--Alexander twists.	

	\begin{definition}
		Let $(a,b)\subset\Real_+$ be an interval of positive real numbers.
		A function $f:(a,b)\to[0,+\infty)$ is said to be \emph{multiplicatively convex} if for all points $t_0,t_1\in(a,b)$
		and every constant $\lambda\in(0,1)$, 
			$$f(t_0^\lambda\cdot t_1^{1-\lambda})\,\leq\, f(t_0)^\lambda\cdot f(t_1)^{1-\lambda}.$$
	\end{definition}
	
	The product of two multiplicatively convex functions is again multiplicatively convex. Furthermore,
	if $f(t)$ is multiplicatively convex, then for any constant $r\in\Real_+$, both $f(t^{\pm r})$ and $f(t)^r$ are
	multiplicatively convex as well.
	
	\begin{lemma}\label{zeroOrNot}
		If a function $f:\Real_+\to [0,+\infty)$ is multiplicatively convex,
		then $f$ is continuous. Moreover, $f$ is either the constant function $0$ or nowhere zero.
	\end{lemma}
	
	\begin{proof}
		If $f$ equals zero at some point $c$, it is clear from the definition that $f$ has to be the constant function $0$.
		When $f$ is nowhere zero, then $\log\circ f\circ \exp$ is a convex function on $\Real$. In either case,
		$f$ is continuous. 
	\end{proof}
	
	\begin{lemma}\label{mMidpointConvex}
		If $f:\Real_+\to [0,+\infty)$ is multiplicatively mid-point convex and upper semi-continuous, namely,
		\begin{itemize}
		\item for every pair of points $t_0,t_1\in\Real_+$, $f(\,\sqrt{t_0t_1}\,)\leq \sqrt{f(t_0)\cdot f(t_1)}$, and
		\item for every point $t_0\in \Real_+$, $\limsup_{t\to t_0} f(t)\leq f(t_0)$,
		\end{itemize}
		then $f$ is multiplicatively convex.
	\end{lemma}
	
	\begin{proof}
		Given any $t_0\in\Real_+$, 
		let $\{t_n\in\Real_+\}_{n\in\Natural}$ be a sequence of points such that $t_n$ converges to $t_0$ and 
		$f(t_n)$ converges to $\liminf_{t\to t_0} f(t)$.
		We have 
		$$f(t_0)^2\leq \limsup_{n\to\infty} f(t_n)f(t_0^2/t_n)\leq\liminf_{t\to t_0}f(t)\cdot\limsup_{t\to t_0} f(t)\leq f(t_0)^2.$$
		Then $\liminf_{t\to t_0}f(t)=\limsup_{t\to t_0}f(t)=f(t_0)$. It follows that $f$ is continuous. It is clear that $f$
		is everywhere positive unless $f$ is constantly zero. When $f$ is everywhere positive,
		we may take $F=\log\circ f\circ \exp$ which is mid-point convex and continuous, so it is well known
		that $F$ is convex, or equivalently, that $f$	is multiplicatively convex.
	\end{proof}
	
	\begin{definition}\label{degree-mConvex-eBounded}
		Let $(a,b)\subset\Real_+$ be an interval of positive real numbers.
		A nowhere zero multiplicatively convex function $f:(a,b)\to(0,+\infty)$ is said to have \emph{bounded exponent} if there exists
		some positive constant $R$
		such that for all pairs of distinct points $t_0,t_1\in(a,b)$, 
			$$\left|\frac{\log f(t_1)-\log f(t_0)}{\log t_1-\log t_0}\right|\,\leq\, R.$$
	\end{definition}
	
	For multiplicatively convex functions, 
	the growth bound degree (Definition \ref{degree-b}) can be characterized by the limit exponents:
	
	\begin{lemma}\label{degree-bCharacterization}
		Suppose that $f:\Real_+\to[0,+\infty)$ is a nowhere zero multiplicatively convex function. 
		Then the growth bound degree 
		$\mathrm{deg}^{\mathtt{b}}(f)\in\Real$ exists
		if and only if $f$ has bounded exponent. Moreover, in this case, the following equalities hold true:
		$$\mathrm{deg}^{\mathtt{b}}_{+\infty}(f)\,=\,\lim_{t_0,t_1\to +\infty}
		\frac{\log f(t_0)-\log f(t_1)}{\log t_0-\log t_1},$$
		and
		$$\mathrm{deg}^{\mathtt{b}}_{0+}(f)\,=\,\lim_{t_0,t_1\to 0+}
		\frac{\log f(t_0)-\log f(t_1)}{\log t_0-\log t_1}.$$	
	\end{lemma}
	
	\begin{proof}
		We show that the equalities hold if $\mathrm{deg}^{\mathtt{b}}\in\Real_+$ exists.
		If there exists $D_{+\infty}\in\Real$ such that 
		$\lim_{t\to+\infty}f(t)\cdot t^{-D_{+\infty}}\,=\,0$, 
		then $\log f(t)$ is less than or equal to $D_{+\infty}\log t$ for all sufficiently large $t\in\Real_+$.
		For all $t_0,t_1\in \Real_+$, by the multiplicative convexity of $f$,
		\begin{eqnarray*}
		\frac{\log f(t_0)-\log f(t_1)}{\log t_0-\log t_1}
		&\leq&\limsup_{t\to +\infty}\frac{\log f(t)-\log f(t_1)}{\log t-\log t_1}\\
		&\leq&\limsup_{t\to +\infty}\frac{D_{+\infty}\log t-\log f(t_1)}{\log t-\log t_1}\\
		&=& D_{+\infty}.
		\end{eqnarray*}
		Denote by 
			$$d_{+\infty}\,=\,\lim_{t_0,t_1\to +\infty}
		\frac{\log f(t_0)-\log f(t_1)}{\log t_0-\log t_1}\,\in\,\Real.$$
		It is easy to see that for any constant $\delta>0$,
		$$\lim_{t\to+\infty} f(t)\cdot t^{-(d_{+\infty}+\delta)}=0.$$
		Consequently,
		$$\mathrm{deg}^{\mathtt{b}}_{+\infty}(f)\,=\,d_{+\infty}.$$
		The equality for $0+$ can be proved in a similar way. 
		We have shown the `only-if' direction.
		The existence of exponent bound leads to the existence of $d_{+\infty}$ and $d_{0+}$ in $\Real$,
		so
		$$d_{0+}-1\,<\mathrm{deg}^{\mathtt{b}}_{0+}(f)\,\leq\, \mathrm{deg}^{\mathtt{b}}_{+\infty}(f)\,<\, d_{+\infty}+1.$$
		This shows the `if' direction.
	\end{proof}

	\begin{example}\ 
	\begin{enumerate}
	\item A \emph{monomial function} on an interval $(a,b)$ is a function of the form $f(t)=Ct^r$ for some constants $C\in \Real_+$
	and $r\in \Real$.	Such a function is \emph{multiplicatively linear} in the sense that for all points
	$t_0,t_1\in(a,b)$ and for every constant $\lambda\in(0,1)$,
			$$f(t_0^{1-\lambda}\cdot t_1^{\lambda})\,=\, f(t_0)^{1-\lambda}\cdot f(t_1)^{\lambda}.$$
	\item A \emph{piecewise monomial function} on an interval $(a,b)$ is a continuous function $f:(a,b)\to(0,+\infty)$ such that for
	finitely many points $a=c_0<c_1<\cdots<c_{n-1}<c_n=b$, the function is a monomial $C_it^{r_i}$ on the subinterval $(c_{i-1},c_i)$
	where $i$ runs over $1,\cdots,n$. Such a continuous function is multiplicatively convex if and only if $r_1\leq r_2\leq\cdots\leq r_n$.
	\item Given any Laurent polynomial 
		$$p(z)\,=\,D\cdot z^n\cdot \prod_{i=1}^{l} (z-b_i)\in \Complex[z,z^{-1}],$$
		with a leading coefficient $D\in\Complex^\times$ and nontrivial zeros $b_i\in \Complex^\times$,
		the function 
			$$M(p(z);\,t)\,=\,|D|\cdot t^{n}\cdot\prod_{i=1}^{l}\max(t,|b_i|),$$
		of the variable $t\in\Real_+$,	is piecewise monomial and multiplicatively convex.
	\end{enumerate}
	\end{example}

\section{Multiplicative convexity and exponent bound}\label{Sec-mConvex-eBounded}
	In this section, we show that residually finite $L^2$--Alexander twists
	result in multiplicatively convex determinant functions with bounded exponents.
	
	\begin{theorem}\label{mConvex-eBounded}
		Given any admissible triple 
		$(\pi,\phi,\gamma)$ over $\Real$ and any square matrix $A$ over $\Complex\pi$, 
		denote by $V:\,\Real_+\to[0,+\infty)$	the regular Fuglede--Kadison determinant function 
		$$V(t)=\mathrm{det}^{\mathtt{r}}_{\mathcal{N}(G)}\left(\kappa(\phi,\gamma,t)(A)\right),$$
		where $G$ is the target group of $\gamma$ 
		and $\kappa(\phi,\gamma,t)$ is the induced change of coefficients.
		
		Suppose that $G$ is finitely generated and residually finite.
		Then $V(t)$ is either constantly zero or 
		multiplicatively convex with exponent bounded.
		Moreover, there exists a constant $R(A,\phi)\in[0,+\infty)$ 
		depending only on $A$ and $\phi$ so that
			$$\mathrm{deg}^{\mathtt{b}}(V)\,\leq\,R(A,\phi).$$
	\end{theorem}
	
	The rest of this section is devoted to the proof of Theorem \ref{mConvex-eBounded}.
	
	\subsection{The degree bound}
	For any $p\times p$ matrix $A$ over $\Complex\pi$, we can decompose $A$ as a unique sum:
		$$A\,=\,\sum_{g\in\pi} g\cdot A_g$$
	where $A_g$ are $p\times p$ matrices over $\Complex$ and only finitely many $A_g$ are nonzero.
	Given any homomorphism $\phi\in \mathrm{Hom}(\pi,\Real)$,
	we define 
		$$R(A,\phi)\,=\,p\cdot\left(\max_{A_g\neq\mathbf{0}} \phi(g)-\min_{A_g\neq\mathbf{0}} \phi(g)\right).$$
	
	The quantity $R(A,\phi)$ behaves well under operations of the matrix and the cohomology class.
	In fact, we observe the following elementary properties. The proof is straightforward so we omit it in this paper.
	\begin{lemma}\
	\begin{enumerate}
		\item For all $A\in\Matrix_{p\times p}(\Complex)\subset\Matrix_{p\times p}(\Complex\pi)$, $R(A,\phi)=0$.
		\item For all $A,B\in\Matrix_{p\times p}(\Complex\pi)$,
			$$R(AB,\phi)\leq R(A,\phi)+R(B,\phi)$$
		and
			$$R(A+B,\phi)\leq \max(R(A,\phi),R(B,\phi))$$
		\item For all $A\in\Matrix_{p\times p}(\Complex\pi)$, and $c\in \Real$, and $\phi,\psi\in \mathrm{Hom}(\pi,\Real)$, 
			$$R(A,c\phi)\,=\,|c|\cdot R(A,\phi)$$
		and
			$$R(A,\phi+\psi)\leq R(A,\phi)+R(A,\psi).$$
		\item Let $\gamma:\pi\to G$ be a group homomorphism. For all $A\in\Matrix_{p\times p}(\Complex\pi)$ and $\xi\in H^1(G;\Real)$,
			$$R(A,\gamma^*\xi)\geq R(\gamma_*A,\xi)$$
	\end{enumerate}
	\end{lemma}
	
	The following lemma can be combined with Lemma \ref{degree-bCharacterization}
	to yield the degree bound, once we have shown that $V(t)$ is multiplicatively convex.	
	
	\begin{lemma}\label{eBoundEstimate}
		Given any admissible triple $(\pi,\phi,\gamma)$ over $\Real$ and any square matrix $A$ over $\Complex\pi$, 
		write
		$$V(t)=\mathrm{det}^{\mathtt{r}}_{\mathcal{N}(G)}\left(\kappa(\phi,\gamma,t)(A)\right)$$
		where $G$ is the target group of $\gamma$. Then the following statement holds true.
		
		For every constant $R'>R(A,\phi)$, there exist constants $D_{+\infty},D_{0+}\in\Real$ such that
		$D_{+\infty}-D_{0+}< R'$. Moreover, 
		$$\lim_{t\to+\infty}V(t)\cdot t^{-D_{+\infty}}\,=\,0,$$
		and 
		$$\lim_{t\to0+} V(t)\cdot t^{-D_{0+}}\,=\,0.$$
	\end{lemma}

	\begin{proof}
		We adopt the notations at the beginning of this subsection. Given $R'>R(A,\phi)$, we denote by $5\delta$
		the difference $R'-R(A,\phi)$. Take
			$$D_{+\infty}\,=\,2\delta+p\cdot\max_{A_g\neq\mathbf{0}}\,\phi(g),$$
		and 
			$$D_{0+}\,=\,-2\delta+p\cdot\min_{A_g\neq\mathbf{0}}\,\phi(g).$$
		For sufficiently large $t\in\Real_+$, the operator norm of $t^{-D_{+\infty}+\delta}\cdot \kappa(\phi,\gamma,t)(A)$ is bounded by
		$1$. Therefore, 
		$$0\,\leq\,\limsup_{t\to+\infty}V(t)\cdot t^{-D_{+\infty}}
		\,\leq\, 1^p\cdot \lim_{t\to+\infty} \mathrm{det}^{\mathtt{r}}_{\mathcal{N}(G)}(t^{-\delta}\cdot\mathbf{1})
		\,=\,\lim_{t\to+\infty}t^{-p\delta}\,=\,0.$$
		This yields the asserted limit for $t\to+\infty$. The limit for $t\to0+$ can be proved in a similar way.
	\end{proof}

	\subsection{Multiplicative convexity for virtually abelian twists}
	In this section, we prove Theorem \ref{mConvex-eBounded} under the assumption that $G$
	is finitely generated and virtually abelian. 
	
	Given an admissible triple $(\pi,\phi,\gamma)$ over $\Real$ and a parameter value $t\in\Real_+$,
	for any $p\times p$ matrix $A$ of $\Complex\pi$, we define
		$$A_G(t)\,=\,\kappa(\phi,\gamma,t)(A)\in\Matrix_{p\times p}(\Complex G)$$
	and write
		$$V(t)\,=\,\mathrm{det}^{\mathtt{r}}_{\mathcal{N}(G)}(A_G(t)).$$
	
	\begin{proposition}\label{mConvex-eBounded-VA}
		Let $(\pi,\phi,\gamma)$ is an admissible triple over $\Real$.
		Suppose that $G$ is finitely generated and virtually abelian.
		Then for every matrix $A\in \Matrix_{p\times p}(\Complex\pi)$, 
		the function $V(t)$ is multiplicatively convex.
	\end{proposition}

	The following lemma treats the essential case where $G$ is finitely generated and free abelian.
	
	\begin{lemma}\label{mConvex-eBounded-FA}
		Let $(\pi,\phi,\gamma)$ be an admissible triple over $\Real$.
		Suppose that $\gamma$ is an isomorphism onto a finitely generated free abelian group $G$.
		Then for every $A\in\Matrix_{p\times p}(\Complex\pi)$,
		the function $V(t)$ is multiplicatively convex.
	\end{lemma}

	\begin{proof}
		For any  admissible triple $(\pi,\phi,\gamma)$ over $\Real$,
		the image $\phi(\pi)$ is finitely generated as $G$ is finitely generated and free abelian.
		Take a basis $r_1,\cdots,r_d\in\Real_+$ of 
		the $\Rational$-vector space spanned by $\phi(\pi)\subset\Real$.
		Possibly after dividing each $r_i$ by a positive integer, we can decompose $\phi$ as a sum:
			$$\phi\,=\,r_1\phi_1+\cdots+r_d\phi_d$$
		where $\phi_i$ are homomorphisms in $\mathrm{Hom}(\pi,\Integral)$.
		We fix such a basis for the rest of the proof.
		Consider a multivariable version of twist as follows.
		Given any vector $\vec{t}=(t_1,\cdots,t_d)\in \Real_+^d$, there is a homomorphism of rings:
		$$\kappa(\phi,\gamma,\vec{t}):\,\Integral\pi\longrightarrow \Real G$$
		defined uniquely by
			$$\kappa(\phi,\gamma,\vec{t})(g)\,=\,t_1^{\phi_1(g)}\cdots t_d^{\phi_d(g)}\gamma(g)$$
		for all $g\in\pi$ via linear extension over $\Integral$.
		There are induced homomorphisms between matrix algebras over $\Complex\pi$ and $\Complex G$ as before.
		We define
			$$A_G(\vec{t})\,=\,\kappa(\phi,\gamma,\vec{t})(A)\,\in\, \Matrix_{p\times p}(\Complex G).$$
		Denote
			$$W(\vec{t})\,=\,\mathrm{det}^{\mathtt{r}}_{\mathcal{N}(G)}\left(A_G(\vec{t})\right).$$
		Then
			$$V(t)\,=\,W((t^{r_1},\cdots,t^{r_d})).$$
		
		On the other hand, we identify $A_G(\vec{t})$ as a family of $p\times p$ matrices over 
		the multivariable Laurent polynomial ring $\Complex[z_1^{\pm1},\cdots,z_l^{\pm1}]$, where $l$ is the rank of $G$.
		Denote by $\vec{1}$ is the diagonal vector $(1,\cdots,1)\in\Real_+^d$.
		If we write the Laurent polynomial matrix at $\vec{1}$ as:
			$$A_G(\vec{1})\,=\,A_G(\vec{1})\,(z_1,\cdots,z_l),$$
		then at $\vec{t}$ the Laurent polynomial matrix can be computed by:
			$$A_{G}(\vec{t})\,=\,A_{G}(\vec{1})\,
			(\tilde{t}_1z_1,\cdots,\tilde{t}_lz_l)$$
		where, for $j$ running over $1,\cdots,l$,
			$$\tilde{t}_j=t_1^{\phi_1(z_j)}\cdot\cdots\cdot t_d^{\phi_d(z_j)}.$$
		In fact, the relation can be checked by looking at the monomials in each entry of $A_{G}(\vec{1})$.
		The effect of the twist is that in any monomial, 
		each $z_j$ that appears contributes an exponent $\phi_i(z_j)$ to the associated coefficient $t_i$.
		
		The value of $W(\vec{t})$ can be computed by the (multiplicative) Mahler measure 
		of the usual determinant of the Laurent polynomial matrix $A_G(\vec{t})$.
		Precisely, the usual determinant gives rise to a Laurent polynomial for the square matrix at $\vec{1}$:
			$$p_A(z_1,\cdots,z_l)\,=\mathrm{Det}_{\Complex[z_1^{\pm1},\cdots,z_l^{\pm1}]}\left(A_G(\vec{1})\right),$$
		so
			$$p_A(\tilde{t}_1z_1,\cdots,\tilde{t}_lz_l)\,=\mathrm{Det}_{\Complex[z_1^{\pm1},\cdots,z_l^{\pm1}]}\left(A_G(\vec{t})\right).$$
		By \cite[Lemma 2.6]{DFL-torsion}, (cf.~\cite[Exercise 3.8]{Lueck-book} and \cite[Section 1.2]{Raimbault}),
		if $p_A$ is not the zero polynomial,
		\begin{eqnarray*}
			W(\vec{t})&=&M(p_A(\tilde{t}_1z_1,\cdots,\tilde{t}_lz_l))\\
			&=&\exp\left[\frac{1}{(2\pi)^l}\cdot\int_0^{2\pi}\cdots\int_0^{2\pi}
			\log\left|p_A(\tilde{t}_1e^{\mathbf{i}\theta_1},\cdots,\tilde{t}_l e^{\mathbf{i}\theta_l})\right|\ud\theta_1,\cdots\ud\theta_l\right].
		\end{eqnarray*}
		
		Note that if $p_A$ is the zero polynomial, then $W(\vec{t})$ and $V(t)$ are constantly zero, so the multiplicative convexity of $V(t)$ holds
			in this trivial case. We assume in the rest of the proof that $p_A$ is not the zero polynomial.
		
		First consider the case when $(\pi,\phi,\gamma)$ is an admissible triple over $\Rational$.
		In this case, $d$ is at most $1$. We can assume that $d$ equals $1$, 
		since otherwise $V(t)$ is a constant function.
		There is a splitting short exact sequence of free abelian groups:
			$$1\longrightarrow \gamma(\mathrm{Ker}(\phi))\longrightarrow G\stackrel{\phi\circ\gamma^{-1}}{\longrightarrow} \phi(\pi)\longrightarrow 1.$$
		We may choose a basis of the free abelian group $G$ such that
		$\phi(z_l)=mr_1$ for some nonzero integer $m$ and $\phi(z_i)=0$ for all other $z_i$.
		For any given values $\theta_1,\cdots,\theta_{l-1}\in[0,2\pi]$, we introduce the notations
			$$q_{\theta_1,\cdots,\theta_{l-1}}(z)\,=\,p_A(e^{\mathbf{i}\theta_1},\cdots,e^{\mathbf{i}\theta_{l-1}},z)\in\Complex[z,z^{-1}],$$
		and
			$$v_{\theta_1,\cdots,\theta_{l-1}}(t)\,=\,\log M(q_{\theta_1,\cdots,\theta_{l-1}}(t^{mr_1}z)).$$
		Then
		\begin{eqnarray*}
			\log V(t)&=&
			\log W(t^{r_1})\\
			&=&
			\frac{1}{(2\pi)^l}\cdot\int_0^{2\pi}\cdots\int_0^{2\pi}
			\log\left|p_A(e^{\mathbf{i}\theta_1},\cdots,e^{\mathbf{i}\theta_{l-1}},t^{mr_1}e^{\mathbf{i}\theta_l})\right|\,\ud\theta_1\cdots\ud\theta_l\\
			&=&\frac{1}{(2\pi)^{l-1}}\cdot\int_0^{2\pi}\cdots\int_0^{2\pi}
			v_{\theta_1,\cdots,\theta_{l-1}}(t)\,\ud\theta_1\cdots\ud\theta_{l-1}.
		\end{eqnarray*}
		For any one-variable Laurent polynomial $q\in\Complex[z,z^{-1}]$, the Mahler measure can be 
		computed using Jensen's formula:
			$$M(q(z))=|D|\cdot\prod_{i=1}^{l}\max(1,|b_i|),$$
		where the constants $D\in\Complex$ and $n\in\Integral$ and $b_i\in\Complex$
		are given by any factorization
			$$q(z)=D\cdot z^n\cdot \prod_{i=1}^{l} (z-b_i)\in \Complex[z,z^{-1}].$$
		It is evident that for any such $q$, 
		the following function in $t\in\Real_+$ is multiplicatively convex:
			$$M(q(t^{mr_1}z))\,=\,
			|D|\cdot t^{nmr_1}\cdot\prod_{i=1}^{l}\max(t^{mr_1},|b_i|),$$
		possibly constantly zero if $q$ is $0$.
		Therefore, for all pairs of distinct points $T_0,T_1\in\Real_+$, and all constants $0<\lambda<1$,
		we have the comparison:
			$$(1-\lambda)\cdot v_{\theta_1,\cdots,\theta_{l-1}}(T_0)+\lambda\cdot v_{\theta_1,\cdots,\theta_{l-1}}(T_1)\,
			\geq\,v_{\theta_1,\cdots,\theta_{l-1}}(T_0^{1-\lambda}\cdot T_1^\lambda).$$
		Integrating both sides and taking the exponential yields
			$$V(T_0)^{1-\lambda}\cdot V(T_1)^\lambda\,\geq\,V(T_0^{1-\lambda}\cdot T_1^\lambda).$$
		In other words, $V(t)$ is multiplicatively convex.

		For the general case over $\Real$, denote by
		$\vec{r}$ the vector $(r_1,\cdots,r_d)\in\Real_+^d.$ Take a sequence of 
		rational vectors $\{\,\vec{r}^{(n)}\in\Rational_+^d\,\}$ which converges to $\vec{r}$ in $\Real^d_+$ as $n$ tends to infinity.
		Observe that for each $\vec{r}^{(n)}$, the function 
			$$V_n(t)\,=\, W((t^{r^{(n)}_1},\cdots,t^{r^{(n)}_d}))$$
		is equal to	the regular Fuglede-Kadison determinant of the matrix
			$$\kappa(\phi^{(n)},\gamma,t)(A)\,\in\,\Matrix_{p\times p}(\Complex G),$$
		where
			$$\phi^{(n)}=r_1^{(n)}\phi_1+\cdots+r_d^{(n)}\phi_d$$
		is a homomorphism in $\mathrm{Hom}(\pi,\Rational)$. Then $V_n(t)$ are multiplicatively convex
		by the rational case that we have proved.
		On the other hand, as $\vec{t}$ varies over $\Real_+^d$, the coefficients of the Laurent polynomials 
		$p_A(\tilde{t}_1z_1,\cdots,\tilde{t}_lz_l)$ varies continuously, so the Mahler measure of the Laurent polynomials
		varies continuously by D.~Boyd \cite{Boyd}.
		In particular, for every $t\in\Real_+$, 
		$$\lim_{n\to\infty} V_n(t)\,=\,V(t).$$
		Given any constants $T_0,T_1\subset\Real_+$ and $0<\lambda<1$, we have shown the multiplicative
		convexity for the rational case:
			$$V_n(T_0)^{1-\lambda}\cdot V_n(T_1)^{\lambda}\geq V_n(T_0^{1-\lambda}\cdot T_1^{\lambda}).$$
		Taking the limit as $n\to\infty$,
			$$V(T_0)^{1-\lambda}\cdot V(T_1)^{\lambda}\geq V(T_0^{1-\lambda}\cdot T_1^{\lambda}).$$
		In other words, the function $V(t)$ is multiplicatively convex.
		This completes the proof.
	\end{proof}

	\begin{proof}[{Proof of Proposition \ref{mConvex-eBounded-VA}}]
		Take a free abelian subgroup $\tilde{G}$ of $\tilde{\pi}$ of finite index, which is hence finitely generated.
		Denote by $\tilde{\pi}$ the preimage $\gamma^{-1}(\tilde{G})$.
		Take restrictions $\tilde{\phi}$, $\tilde{\gamma}$ 
		of given homomorphisms to $\tilde{\pi}$ accordingly.
		The restriction of $A$ to $\Complex\tilde{\pi}$, denoted as $\mathrm{res}^{\tilde\pi}_\pi A$,
		is a square matrix over $\Complex\tilde{\pi}$	of size $p\cdot[\pi:\tilde{\pi}]$.
		We observe that the operation of restriction commutes with $\kappa(\gamma,\phi,t)$ and $*$.
		Denote by $\tilde{V}(t)$ the corresponding determinant function for the admissible triple $(\tilde{\pi},\tilde{\phi},\tilde{\gamma})$
		and the matrix $\mathrm{res}^{\tilde\pi}_\pi A$.
		By basic properties of regular Fuglede--Kadison determinants,
		\begin{eqnarray*}
			V(t)&=&\mathrm{det}^\mathtt{r}_{\mathcal{N}(G)}(A_G(t))\\
			&=&\mathrm{det}^\mathtt{r}_{\mathcal{N}(\gamma(\pi))}\left(\mathrm{res}^{\gamma(\pi)}_G\,(A_G(t))\right)\\
			&=&\mathrm{det}^\mathtt{r}_{\mathcal{N}(\tilde{G})}\left(\mathrm{res}^{\tilde{G}}_G\,(A_G(t))\right)^{1/[\gamma(\pi):\tilde{G}]}\\
			&=&\mathrm{det}^\mathtt{r}_{\mathcal{N}(\tilde{G})}\left((\mathrm{res}^{\tilde\pi}_\pi A)_{\tilde{G}}(t)\right)^{1/[\pi:\tilde{\pi}]}\\
			&=&\tilde{V}(t)^{1/[\pi:\tilde{\pi}]}.
		\end{eqnarray*}
		
		Note that $\tilde{V}(t)$ is constantly zero
		if and only if $V(t)$ is constantly zero. Suppose that $\tilde{V}(t)$ is not constantly zero.
		By Lemma \ref{mConvex-eBounded-VA}, the function $\tilde{V}(t)$ is multiplicatively convex,
		so $V(t)$ is multiplicatively convex as well.
		This completes the proof.
	\end{proof}

	\subsection{Multiplicative convexity for residually finite twists}
	
	Let $(\pi,\phi,\gamma)$ be an admissible triple over $\Real$.
	Suppose that the target group $G$ of $\gamma$ is finitely generated and residually finite. 
	Take a cofinal tower of normal finite index subgroups of $G$:
		$$G\geq N_1\geq N_2\geq \cdots\geq N_n\geq\cdots.$$
	Here the tower being cofinal means that
		$$\bigcap_{n=1}^\infty N_n\,=\,\{\,\mathrm{id}_G\,\}.$$
	Fix a homomorphism $G\to \Real$ via which $\phi$ factors through $\gamma$. 
	Denote by $K_n$ the kernel of $N_n\to H_1(N_n;\Rational)$, which remains normal in $G$. 
	Let
		$$\Gamma_n\,=\,G\,/\,K_n.$$
	There are induced homomorphisms by the composition of $\gamma$ and the quotient $G\to\Gamma_n$,
	denoted as
		$$\gamma_n:\pi\to\Gamma_n.$$
	It is clear that $\Gamma_n$ are all finitely generated and virtually abelian.
	Therefore, we obtain a tower of admissible triples over $\Real$:
		$$\{(\pi,\phi,\gamma_n)\}_{n\in\Natural}$$
	with finitely generated virtually abelian targets.
	
	Given any $p\times p$ matrix $A$ over $\Complex\pi$, and any value of parameter $T\in\Real_+$,
	and any constant $\epsilon\in[0,+\infty)$,
	we introduce a positive operator on $\ell^2(\Gamma_n)^{\oplus p}$:
		$$H_{n,\epsilon}(T)\,=\,\left(\kappa(\phi,\gamma_n,T)(A)\right)^*\left(\kappa(\phi,\gamma_n,T)(A)\right)
		+\epsilon\cdot\mathbf{1}$$
	which is expressed as a $p\times p$ matrix over $\Complex\Gamma_n$.
	When the subscript $n$ is replaced with the symbol $\infty$,
	we adopt the convention that $\Gamma_\infty=G$ and $\gamma_\infty=\gamma$.

	\begin{proof}[Proof of Theorem \ref{mConvex-eBounded}]
		Given an admissible triple $(\pi,\phi,\gamma)$ over $\Real$ and a square matrix $A$ over $\Complex\pi$.
		We adopt the assumptions and notations of this subsection. Possibly after replacing $G$ with the image of $\gamma$,
		which does not affect the value of the determinant,
		we may further assume that $\gamma$ is surjective.
		Then there are uniquely induced homomorphisms $\gamma_{n*}\phi\in\mathrm{Hom}(\Gamma_n,\Real)$
		whose pull-backs through $\gamma$ are $\phi$, and $(\Gamma_n,\mathrm{id}_{\Gamma_n},\gamma_{n*}\phi)$ are admissible triples.
		For parameters $s,T,t\in \Real_+$, we write
			$$W_{n,\epsilon}(s,T)\,=\,\mathrm{det}^{\mathtt{r}}_{\mathcal{N}(\Gamma_n)}\left(\,\kappa(\gamma_{n*}\phi,\mathrm{id}_{\Gamma_n},s)(H_{n,\epsilon}(T))\,\right),$$
		and
			$$V_n(t)\,=\,\mathrm{det}^{\mathtt{r}}_{\mathcal{N}(\Gamma_n)}(\,\kappa(\phi,\gamma_n,t)(A)\,).$$
		Observe that $\kappa(\gamma_{n*}\phi,\mathrm{id}_{\Gamma_n},s)\circ\kappa(\phi,\gamma_n,t)$ equals
		$\kappa(\phi,\gamma_n,st)$.
		Therefore, for any given $T_0,T_1\in\Real_+$, we have the relations:
			$$W_{n,0}(1,\sqrt{T_0T_1})\,=\,V_n(\sqrt{T_0T_1})^2$$
		and
			$$W_{n,0}(\sqrt{T_1/T_0},\sqrt{T_0T_1})\,=\,V_n(T_0)V_n(T_1),$$
		which hold for both $n\in\Natural$ and $\infty$. 
		Note that $W_{n,\epsilon}(1,T)$ is always the regular Fuglede--Kadison determinant
		for a positive operator, but the twisted matrix in the expression of $W_{n,\epsilon}(s,T)$ 
		is not self-adjoint in general.
			
		We claim that the following comparison holds for all $s,T\in\Real_+$:
			$$W_{\infty,0}(1,T)\,\leq\,W_{\infty,0}(s,T).$$
		In fact, by Lemma \ref{mConvex-eBounded-VA}, the function $W_{n,\epsilon}(s,T)$ is multiplicatively convex
		in $s\in\Real_+$ for all $n\in\Natural$ and $\epsilon\in[0,+\infty)$. Observe that $H_{n,\epsilon}(T)$ is self-adjoint,
		so the anti-commutativity of $\kappa(\phi,\gamma_n,s)$ and $*$ yields
		$W_{n,\epsilon}(s,T)=W_{n,\epsilon}(s^{-1},T)$. This implies that for all $\epsilon\in[0,+\infty)$ and $n\in\Natural$,
			$$W_{n,\epsilon}(1,T)\,\leq\,W_{n,\epsilon}(s,T).$$
		Given any arbitrary $\epsilon>0$, Lemma \ref{stable-semicontinuous} and the above imply
		\begin{eqnarray*}
			W_{\infty,\epsilon}(1,T)&=&\lim_{n\to\infty}\,W_{n,\epsilon}(1,T)\\
			&\leq& \limsup_{n\to\infty}\, W_{n,\epsilon}(s,T)\\
			&\leq& W_{\infty,\epsilon}(s,T).
		\end{eqnarray*}
		As $\epsilon$ tends to $0+$, Lemma \ref{norm-semicontinuous} and the above imply
		\begin{eqnarray*}
			W_{\infty,0}(1,T)&=&\lim_{\epsilon\to0+}\,W_{\infty,\epsilon}(1,T)\\
			&\leq& \limsup_{\epsilon\to0+}\, W_{\infty,\epsilon}(s,T)\\
			&\leq& W_{\infty,0}(s,T).
		\end{eqnarray*}
		This proves the claim.
		
		Note that the family of operators $\kappa(\phi,\gamma,s)(A)$ is continuous in $s\in\Real_+$ with respect
		to the norm topology. Lemma \ref{norm-semicontinuous} implies that $V_\infty(t)$ is upper semicontinuous
		in $t\in\Real_+$. On the other hand, the claim implies that $V_\infty(t)$ is multiplicatively mid-point convex
		in $t\in\Real_+$. By Lemma \ref{mMidpointConvex}, the function $V_\infty(t)$, or $V(t)$ as in the statement
		of Theorem \ref{mConvex-eBounded}, is multiplicatively convex.
		
		Provided with the multiplicative convexity, assuming that $V(t)$ is nowhere zero,
		the exponent bound and the degree estimate 
			$$\mathrm{deg}^{\mathtt{b}}(V)\leq R(A,\phi)$$
		follow from Lemma \ref{eBoundEstimate} and Lemma \ref{degree-bCharacterization}.
		This completes the proof of Theorem \ref{mConvex-eBounded}.
	\end{proof}

\section{Continuity of degree}\label{Sec-continuityOfDegree}
	In this section, we show that the growth bound degree of the regular Fuglede--Kadison determinant
	of $L^2$--Alexander twists varies continuously as we deform the cohomology class.

	\begin{theorem}\label{continuityOfDegree}
		Given any admissible triple 
		$(\pi,\phi,\gamma)$ over $\Real$ and any square matrix $A$ over $\Complex\pi$, 
		denote by $G$ the target group of $\gamma$.
		For any vector $\xi\in H^1(G;\,\Real)$, denote by 
		$$V_\xi(t)=\mathrm{det}^{\mathtt{r}}_{\mathcal{N}(G)}\left(\kappa(\phi+\gamma^*\xi,\gamma,t)(A)\right)$$
		the determinant function of $A$ associated
		with the deformed admissible triple $(\pi,\phi+\gamma^*\xi,\gamma)$.
		
		Suppose that $G$ is finitely generated and residually finite.
		Then the function $V_\xi(t)$ is constantly zero at every vector $\xi\in H^1(G;\,\Real)$ whenever 
		it is constantly zero somewhere. Apart from that exception, 
		for all pairs of vectors $\xi,\eta\in H^1(G;\Real)$, 
			$$|\mathrm{deg}^{\mathtt{b}}(V_\xi)-\mathrm{deg}^{\mathtt{b}}(V_\eta)|\,\leq\,2R(A,\gamma^*(\xi-\eta)).$$
		In particular,
		the assignment with the degree $\xi\,\mapsto\,\mathrm{deg}^{\mathtt{b}}(V_\xi(t))$
		defines a Lipschitz continuous function on $H^1(G;\,\Real)$ valued in $[0,+\infty)$.
	\end{theorem}
	
	The continuity of degree is a consequence of Theorem \ref{mConvex-eBounded}.
	The rest of this section is devoted to the proof of Theorem \ref{continuityOfDegree}.
	
	We may assume
	without loss of generality that $\eta\in H^1(G;\,\Real)$ is trivial.
	In fact, otherwise we can replace
	the reference class $\phi$ by $\phi+\gamma^*\eta$.
	Hence $\xi$ and $\eta$ are replaced by $\xi-\eta$ and $0$ respectively.
	
	We adopt the following notations. Given any matrix $A\in\Matrix_{p\times p}(\Complex\pi)$, denote
		$$A_G(t)\,=\,\kappa(\phi,\gamma,t)(A)\,\in\,\Matrix_{p\times p}(\Complex G).$$
	For any vector $\xi\in H^1(G;\Real)\cong\mathrm{Hom}(G;\,\Real)$, we consider the canonical admissible triple
	$(G,\xi,\mathrm{id}_G)$, so for every constant $s\in \Real_+$, there is 
	a matrix deformed from $A_G(t)$, namely:
		$$A_G(t,s)\,=\,\kappa(\xi,\mathrm{id}_G,s)(A_G(t))\,\in\,\Matrix_{p\times p}(\Complex G).$$
	We introduce
		$$W(t,s)\,=\,\mathrm{det}^{\mathtt{r}}_{\mathcal{N}(G)}\left(A_G(t,s)\right).$$
	Note that 
		$$W(t,1)=V_{0}(t)$$
	and
		$$W(t,t)=V_\xi(t).$$
	
	\begin{lemma}\label{alwaysConstantlyZero}
		If the function $V_0(t)$ is constantly zero, then for all vectors
		$\xi\in H^1(G;\,\Real)$, the function $V_\xi(t)$ is constantly zero as well.
	\end{lemma}
	
	\begin{proof}
		Suppose $V_0(t)$ is constantly zero. Given any constant $T_0\in\Real_+$, apply
		Theorem \ref{mConvex-eBounded} to the family of matrices $A_G(T_0,s)$,
		we see that $W(T_0,s)$ is multiplicatively convex in the parameter $s\in\Real_+$. At $s=1$,
		we have $W(T_0,1)=V_0(T_0)=0$. This implies that $W(T_0,s)$ is constantly zero in $s$ by 
		Lemma \ref{zeroOrNot}. In particular, $V_\xi(T_0)=W(T_0,T_0)=0$. As $T_0\in\Real_+$
		is arbitrary, it follows that $V_\xi(t)$ is constantly zero.
	\end{proof}
	
	Now it suffices to assume that the functions $V_\xi(t)$ are nowhere zero,
	for all $\xi\in H^1(G;\,\Real)$.
	By Theorem \ref{mConvex-eBounded}, $V_\xi(t)$ are multiplicatively convex and
	have bounded exponent.
	
	\begin{lemma}\label{topAndBottomEstimates}\ 
	\begin{enumerate}
		\item $|\mathrm{deg}^{\mathtt{b}}_{+\infty}(V_\xi)-\mathrm{deg}^{\mathtt{b}}_{+\infty}(V_0)|\,\leq\,R(A,\,\gamma^*\xi)$;
		\item $|\mathrm{deg}^{\mathtt{b}}_{0+}(V_\xi)-\mathrm{deg}^{\mathtt{b}}_{0+}(V_0)|\,\leq\,R(A,\,\gamma^*\xi)$.
	\end{enumerate}
	\end{lemma}
	
	\begin{proof}
		We prove the first estimate and the second can be proved in the same way.
		
		Given any constant $T_0\in\Real_+$ and $K>0$, it follows from the multiplicative convexity of $W(T_0^{1+K},s)$ in 
		the parameter $s\in\Real_+$ that
		$$\left|\frac{\log W(T_0^{1+K},T_0^{1+K})-\log W(T_0^{1+K},1)}{\log T_0^{1+K} -\log 1}\right| \,\leq\,
		R(A_G(T_0^{1+K}),\xi)\leq R(A,\gamma^*\xi),$$
		so
		$$\left|\log W(T_0^{1+K},T_0^{1+K})-\log W(T_0^{1+K},1)\right| \,\leq\,
		R(A,\gamma^*\xi)\cdot(1+K)\log T_0.$$
		Similarly,
		$$\left|\log W(T_0,T_0)-\log W(T_0,1)\right| \,\leq\,
		R(A,\gamma^*\xi)\cdot\log T_0.$$
		By the multiplicative convexity of $W(t,1)=V_0(t)$, for any arbitrary $\delta>0$,
		the following estimate holds for sufficiently large $T_0>1$ and any arbitrary $K>0$:
		$$\left|\frac{\log W(T_0^{1+K},1)-\log W(T_0,1)}{\log T_0^{1+K} -\log T_0}-\mathrm{deg}^{\mathtt{b}}_{+\infty}(V_0)\right| \,<\,
		\delta,$$
		so
		$$\left|\log W(T_0^{1+K},1)-\log W(T_0,1)-\mathrm{deg}^{\mathtt{b}}_{+\infty}(V_0)K\log T_0\right|<\delta\cdot K\log T_0.$$
		Therefore, for any arbitrary $\delta>0$, the following estimate holds for sufficiently large $T_0>1$ and 
		any arbitrary $K>0$:
		\begin{eqnarray*}
		&&\left|\log W(T_0^{1+K},T_0^{1+K})-\log W(T_0,T_0)-\mathrm{deg}^{\mathtt{b}}_{+\infty}(V_0)K\log T_0\right|\\
		&<&R(A,\gamma^*\xi)\cdot(2+K)\log T_0+\delta\cdot K\log T_0,
		\end{eqnarray*}
		or equivalently,
		$$\left|\frac{\log V_\xi(T_0^{1+K})-\log V_\xi(T_0)}{\log T_0^{1+K}-\log T_0}-\mathrm{deg}^{\mathtt{b}}_{+\infty}(V_0)\right|\\
		\,<\,R(A,\gamma^*\xi)\cdot(1+\frac{2}{K})+\delta.
		$$
		Take the limit as $T_0\to+\infty$, and then take the limit as $K\to +\infty$:
		$$\left|\mathrm{deg}^{\mathtt{b}}_{+\infty}(V_\xi)-\mathrm{deg}^{\mathtt{b}}_{+\infty}(V_0)\right|\,\leq\,R(A,\gamma^*\xi)+\delta.$$
		As $\delta>0$ is an arbitrary constant, the estimate 
		$$|\mathrm{deg}^{\mathtt{b}}_{+\infty}(V_\xi)-\mathrm{deg}^{\mathtt{b}}_{+\infty}(V_0)|\,\leq\,R(A,\,\gamma^*\xi)$$
		follows. The second estimate can be done similarly using $1/T_0$ instead of $T_0$.
	\end{proof}
	
	Combining the estimates of Lemma \ref{topAndBottomEstimates}, we obtain
	$$|\mathrm{deg}^{\mathtt{b}}(V_\xi)-\mathrm{deg}^{\mathtt{b}}(V_0)|\,\leq\,2R(A,\gamma^*\xi).$$
	This completes the proof of Theorem \ref{continuityOfDegree}.
	
\section{Asymptotics for integral matrices}\label{Sec-asymptotics}
	In this section, we give a criterion for checking under special circumstances that 
	the regular Fuglede--Kadison determinant of $L^2$--Alexander twists is asymptotically monomial.
	
	\begin{definition}
		Let $(\pi,\gamma,\phi)$ be an admissible triple with a countable target group $G$,
		and 
		$$G\to\cdots\to\Gamma_n\to\cdots\to\Gamma_2\to\Gamma_1$$
		be a cofinal tower of quotients of $G$.
		Denote by $\psi_n:G\to\Gamma_n$ the quotient homomorphisms.
		A sequence of admissible triples 
		$$\{(\pi,\gamma_n,\phi)\}_{n\in\mathbb{N}}$$
		with target groups $\{\Gamma_n\}_{n\in\mathbb{N}}$ is said to
		form a \emph{cofinal tower of quotients} of $(\pi,\gamma,\phi)$
		if $\gamma_n=\psi_n\circ\gamma$ holds
		for every	$n\in\Natural$.
		For simplicity,
		we often speak of cofinal towers of admissible triples
		without explicitly mentioning the cofinal tower of quotients of $G$.
	\end{definition}
	
	In the statement of the theorem below, we adopt the notation
		$$V_n(t)=\mathrm{det}^{\mathtt{r}}_{\mathcal{N}(\Gamma_n)}\left(\kappa(\phi,\gamma_n,t)(A)\right).$$
	The notation $V_G(t)$ is understood similarly.
	
	\begin{theorem}\label{rationalAsymptotic}
		Let $(\pi,\gamma_G,\phi)$ be an admissible triple over $\Real$ 
		with a finitely generated target group $G$.
		Let $A$ be a square matrix over $\Integral\pi$.
		
		Suppose that there exists a sequence of admissible triples $\{(\pi,\gamma_n,\phi)\}_{n\in\Natural}$ 
		over $\Real$ satisfying all the following conditions:
		\begin{itemize}
		\item The target groups $\Gamma_n$ of $\gamma_n$ 
		are finitely generated and virtually abelian.
		\item The sequence of admissible triples $\{(\pi,\gamma_n,\phi)\}_{n\in\Natural}$ 
		forms a cofinal tower of quotients of $(\pi,\gamma_G,\phi)$.
		\item The sequence of degrees $\{\mathrm{deg}^{\mathtt{b}}(V_n)\}_{n\in\Natural}$
		converges to $\mathrm{deg}^{\mathtt{b}}(V_G)$ in $[0,+\infty)$.
		\end{itemize}
		In particular, note that $V_G(t)$ should not be constantly zero.
		Then, as $t\to+\infty$,
			$$V_G(t)\,\sim\,C_{+\infty}\cdot t^{\mathrm{deg}^{\mathtt{b}}_{+\infty}(V_G)}$$
		for some constant
			$$C_{+\infty}\in[1,\,V_G(1)].$$
		The same statement holds true with $+\infty$ replaced by $0+$.
	\end{theorem}

	We point out that among the three conditions the convergence of degrees is usually the hardest
	to satisfy or to verify. The $\Integral\pi$--matrix assumption is responsible for the lower bound
	$1$ of the coefficients $C_{+\infty}$ and $C_{0+}$ in an essential way. 
	In particular, the argument does not apply to matrices over $\Complex\pi$ to yield 
	similar monomial asymptoticity.
	
	The rest of this section is devoted to the proof of \ref{rationalAsymptotic}.

	\begin{lemma}\label{mConvexVersion}
		Let $\hat{f}$ be a nowhere zero multiplicatively convex function on $\Real_+$ with bounded exponent.
		Suppose that there exists a sequence of nowhere zero multiplicatively convex functions
		on $\Real_+$ with bounded exponent $\{f_n\}_{n\in\Natural}$ satisfying all the following conditions:
		\begin{itemize}
		\item There exists a uniform constant $L\in\Real$ such that for all $n\in\Natural$ and for all pairs of distinct
		points $t_0,t_1\in\Real_+$,
			$$\frac{\log f_n(t_0)\log t_1-\log f_n(t_1)\log t_0}{\log t_1-\log t_0}\,\geq\, L.$$
		\item For every point $t\in\Real_+$,
			$$\limsup_{n\to\infty} f_n(t)\leq\hat{f}(t).$$
		\item
			$$\lim_{n\to\infty}\mathrm{deg}^{\mathtt{b}}(f_n)\,=\,\mathrm{deg}^{\mathtt{b}}(\hat{f}).$$
		\end{itemize}
		Then as $t\to+\infty$, 
			$$\hat{f}(t)\,\sim\,C_{+\infty}\cdot t^{\mathrm{deg}^{\mathtt{b}}_{+\infty}(\hat{f})}$$
		for some constant 
			$$C_{+\infty}\in [e^L,\,\hat{f}(1)].$$
		The same statement holds true with $+\infty$ replaced by $0+$.
	\end{lemma}
		
	\begin{proof}
		To understand the geometric meaning of the terms in presence, consider the log--log plot of a
		function $f:\,\Real_+\to\Real_+$,
		namely, the parametrized curve
			$$\mathcal{P}_f(t)\,=\,(\log t,\,\log f(t)),\,t\in\Real_+$$
		on the Cartesian XY plane. 
		The line through a pair of distinct points $\mathcal{P}_f(t_0)$ and $\mathcal{P}_f(t_1)$ has the slope
			$$\alpha_f(t_0,t_1)\,=\,\frac{\log f(t_1)-\log f(t_0)}{\log t_1-\log t_0},$$
		and it has the Y-intercept
			$$\beta_f(t_0,t_1)\,=\,\frac{\log f(t_0)\log t_1-\log f(t_1)\log t_0}{\log t_1-\log t_0}.$$
		If $f$ is multiplicatively convex with bounded exponent, then $\mathcal{P}_f$ is 
		a convex graph. The constants $\mathrm{deg}^{\mathtt{b}}_{+\infty}(f)$ 
		and $\mathrm{deg}^{\mathtt{b}}_{0+}(f)$ are exactly the supremum and 
		the infimum for slope of chords of $\mathcal{P}_f$, respectively, (Lemma \ref{degree-bCharacterization}).
		For any such $f$, it is easy to see that in as $t\to+\infty$, 
		the asymptotic formula
			$$f(t)\sim C_{+\infty}\cdot t^{\mathrm{deg}^{\mathtt{b}}_{+\infty}(f)}$$
		holds for some constant $C_{+\infty}\in\Real_+$ if and only if the following limit exists in $\Real$:
			$$\beta_{+\infty}(f)\,=\,\lim_{t_0,t_1\to+\infty}\beta_f(t_0,t_1),$$
		(which otherwise diverges to $-\infty$). Moreover, $\log C_{+\infty}$ must be $\beta_{+\infty}(f)$ if the asymptotic formula holds.
		The same criterion holds for $0+$ in place of $+\infty$. 
		We also observe that if $\beta_f(t_0,t_1)$ is uniformly bounded below by some constant $L\in\Real$
		for all pairs of distinct parameters $t_0,t_1\in\Real_+$, then equivalently,
		the curve $\mathcal{P}_f$
		is contained entirely in the wedge region $\mathcal{V}(L,f)$ supported on the two rays
		emanating from the point $(0,L)$ along the directions
		$(-1,-\mathrm{deg}^{\mathtt{b}}_{0+}(f))$ and $(1,\mathrm{deg}^{\mathtt{b}}_{+\infty}(f))$. 
		
		To prove Lemma \ref{mConvexVersion}, we observe from the geometric meaning that 
		the limit Y-intercept $C_{+\infty}$ is at most $\hat{f}(1)$.		
		It remains to bound $C_{+\infty}$ from below by $e^L$,
		or equivalently, to show that the log--log plot of the function $\hat{f}$ is contained
		in the wedge region $\mathcal{V}(L,\hat{f})$.
		
		We argue by contradiction, supposing that there
		were a point $P=\mathcal{P}_{\hat{f}}(T_0)$ lying outside $\mathcal{V}(L,\hat{f})$.
		By the first condition, the curves $\mathcal{P}_n$ of $f_n$ are all contained in
		their own wedge regions $\mathcal{V}(L,f_n)$. In particular, 
		the second condition implies that $T_0\neq1$.
		Let $3\delta\cdot |\log T_0|$ be the vertical distance of 
		$P$ from $\mathcal{V}(L,\hat{f})$.
		For all sufficiently large $n$, the second condition implies that the right side of $\mathcal{V}(L,f_n)$
		is at most $\delta \cdot|\log T_0|$ above $P$. Then the third condition forces 
		the slope of the left side of $\mathcal{V}(L,f_n)$ to be at least 
		$\delta$ less than that of $\mathcal{V}(L,\hat{f})$ for all sufficiently large $n$.
		Consequently, for some parameter value $T_1\in\Real_+$ that is sufficiently close to $0+$, 
		the curve point $Q=\mathcal{P}_{\hat{f}}(T_1)$ 
		must stay uniformly below the left sides of all those $\mathcal{V}(L,f_n)$,
		for instance, of distance at least $1$.
		However, we see that the second condition is violated at the point $Q$: We have shown
		that the curves $\mathcal{P}_{n}$ 
		would have been at least distance $1$ above $Q$ for all sufficiently large $n$.
		The contradiction completes the proof.
	\end{proof}

	\begin{lemma}\label{coefficientVA}
		Let $(\pi,\phi,\gamma)$ be an admissible triple over $\Real$ with a target group $G$.
		Let $A$ be a square matrix over $\Integral\pi$.	
		Suppose that $G$ is finitely generated and virtually abelian.
		Then for all pairs of distinct
		points $t_0,t_1\in\Real_+$,
			$$\frac{\log V_G(t_0)\log t_1-\log V_G(t_1)\log t_0}{\log t_1-\log t_0}\,\geq\, 0,$$
		unless $V_G(t)$ is constantly zero. 
	\end{lemma}
	
	\begin{proof}
		By Theorem \ref{mConvex-eBounded}, the function $V_G(t)$ is either constantly zero or
		multiplicatively convex with bounded exponent. It suffices to consider the latter case.
		By the geometric meaning of the expression explained in the proof of Lemma \ref{mConvexVersion},
		we can equivalently prove that $V_G(t)$ is asymptotically monomial in both ends with the coefficient 
		no less than $1$.		
		
		We start by a few reductions. Observe that whether or not the asserted inequality holds true
		does not change under passage from $G$ to any finite index subgroup $\tilde{G}$ of $\gamma(\pi)$.
		Indeed, by basic properties of regular Fuglede--Kadison determinants,
		\begin{eqnarray*}
			V_G(t)&=&\mathrm{det}^{\mathtt{r}}_{\mathcal{N}(G)}(\kappa(\gamma,\phi,t)(A))\\
			&=&\mathrm{det}^{\mathtt{r}}_{\mathcal{N}(\gamma(\pi))}(\kappa(\gamma,\phi,t)(A))\\
			&=&\mathrm{det}^{\mathtt{r}}_{\mathcal{N}(\tilde{G})}
			\left(\kappa(\gamma,\phi,t)(\mathrm{res}_{\gamma(\pi)}^{\tilde{G}}(A))\right)^{\frac{1}{[\gamma(\pi):\tilde{G}]}}\\
			&=& V_{\tilde{G}}(t)^{\frac{1}{[\gamma(\pi):\tilde{G}]}}.
		\end{eqnarray*}
		Therefore, possibly after replacing $G$ with a finite index subgroup $\tilde{G}$ of $\gamma(\pi)$,
		and replacing $\pi$ with $\gamma(\pi)$, we may assume 
		without loss of generality
		that $\gamma$ is an isomorphism,
		and $G$ is a finitely generated free abelian group.
		
		After these reductions, we denote by $l$ the rank of $G$ and 
		identify $\Complex G$ with the Laurent polynomial ring
		$\Complex[z_1^{\pm1},\cdots,z_l^{\pm1}]$. 
		Choose a basis $r_1,\cdots,r_d\in\Real_+$ of
		the $\Rational$-vector space spanned by $\phi(\pi)$
		such that elements of $\phi(\pi)$ are $\Integral$-linear combinations of $r_i$.
		Then we can uniquely decompose $\phi$ as a sum:
			$$\phi\,=\,r_1\phi_1+\cdots+r_d\phi_d$$
		where $\phi_i$ are homomorphisms in $\mathrm{Hom}(\pi,\Integral)$.
		
		As in the proof of Lemma \ref{mConvex-eBounded-FA}, the function $V_G(t)$ can be
		expressed in terms of a multivariable determinant function:
			$$V_G(t)\,=\,W((t^{r_1},\cdots,t^{r_d})),$$
		where for any vector $\vec{t}=(t_1,\cdots,t_d)\in \Real_+^d$,		
		\begin{eqnarray*}
			W(\vec{t})&=&\mathrm{det}^{\mathtt{r}}_{\mathcal{N}(G)}\left(A_G(\vec{t})\right)\\
			&=&M(p_A(\tilde{t}_1z_1,\cdots,\tilde{t}_lz_l))\\
			&=&\exp\left[\frac{1}{(2\pi)^l}\cdot\int_0^{2\pi}\cdots\int_0^{2\pi}
			\log(|p_A(\tilde{t}_1e^{\mathbf{i}\theta_1},\cdots,\tilde{t}_le^{\mathbf{i}\theta_l})|)\ud\theta_1,\cdots\ud\theta_l\right],
		\end{eqnarray*}
		and for each $j$,
			$$\tilde{t}_j=t_1^{\phi_1(z_j)}\cdot\cdots\cdot t_d^{\phi_d(z_j)}.$$
		Recall the notations there that the Laurent polynomial matrix
			$$A_G(\vec{t})\,=\,\kappa(\phi,\gamma,\vec{t})(A)\,\in\, \Matrix_{p\times p}(\Complex[z_1^{\pm1},\cdots,z_l^{\pm1}]).$$
		is defined using the homomorphism of matrix algebras $\kappa(\phi,\gamma,\vec{t})$ determined by the formula 
			$$\kappa(\phi,\gamma,\vec{t})(g)\,=\,t_1^{\phi_1(g)}\cdots t_d^{\phi_d(g)}\gamma(g)$$
		for all $g\in\pi$.			
		The usual determinant of the Laurent polynomial matrix 
		$A_G(\vec{t})$ at the diagonal vector $\vec{1}=(1,\dots,1)\in\Integral^d$ gives rise to the Laurent polynomial
			$$p_A(z_1,\cdots,z_l)\,=\mathrm{Det}_{\Complex[z_1^{\pm1},\cdots,z_l^{\pm1}]}\left(A_G(\vec{1})\right).$$
				
		The idea is to govern the asymptotics of $V_G(t)$ by the fact that $p_A$ is a Laurent polynomial over $\Integral$,
		since $A$ is assumed to be over $\Integral\pi$.
		To this end, expand the Laurent polynomial $p_A$ as
			$$p_A(z_1,\cdots,z_l)\,=\,\sum_{\vec{v}\in\Integral^l}a_{\vec{v}}z_1^{v_1}\cdots z_l^{v_l}$$
		where $v_i$ are the entries of $\vec{v}\in\Integral^l$.
		Only finitely many coefficients $a_{\vec{v}}$ in the summation are nonzero. 
		For any vector $\vec{v}\in\Integral^l$, denote 
			$$\Phi\vec{v}\,=\,(\phi_1(z_1^{v_1}\cdots z_l^{v_l}),\cdots,\phi_d(z_1^{v_1}\cdots z_l^{v_l}))\,\in\,\Integral^d.$$
		Denote by $\vec{r}\in\Real_+^d$ the vector $(r_1,\cdots,r_d)$. 
		Let $\vec{w}_{\mathtt{top}}\in\Integral^d$ be the unique vector
		at which the maximum of the following set is achieved:
			$$\left\{\,\langle\, \vec{r},\,\vec{w}\,\rangle\in\Real\,:\sum_{\Phi\vec{v}=\vec{w}}\,a_{\vec{v}}\neq0\,\right\}.$$
		The uniqueness is a consequence of the linear independence of $r_1,\cdots,r_d$ over $\Rational$.
		The integrand for $V_G(t)$, denoted as $\omega(t,\vec{\theta})$, can be calculated by:
		\begin{eqnarray*}
		\omega(t,\vec{\theta})&=&
		\log\left|p_A(t^{r_1\phi_1(z_1)+\cdots+r_d\phi_d(z_1)}e^{\mathbf{i}\theta_1},\cdots,t^{r_1\phi_1(z_l)+\cdots+r_d\phi_d(z_l)}e^{\mathbf{i}\theta_l})\right|\\
		&=&
			\log\left|\sum_{\vec{w}\in\Integral^d}\sum_{\Phi\vec{v}=\vec{w}}a_{\vec{v}}
			\,t^{\langle \vec{r},\,\Phi\vec{v}\rangle}e^{\mathbf{i}\vec{\langle\theta},\vec{v}\rangle}\right|\\
			&=&
			\log\left|
			\sum_{\Phi\vec{v}=\vec{w}_{\mathtt{top}} }a_{\vec{v}}
			\,t^{\langle \vec{r},\,\vec{w}_{\mathtt{top}}\rangle}e^{\mathbf{i}\vec{\langle\theta},\vec{v}\rangle}
			+
			\sum_{\Phi\vec{v}\neq\vec{w}_{\mathtt{top}}}a_{\vec{v}}
			\,t^{\langle \vec{r},\,\Phi\vec{v}\rangle}e^{\mathbf{i}\vec{\langle\theta},\vec{v}\rangle}\right|\\
			&=&
			\log\left|
			\sum_{\Phi\vec{v}=\vec{w}_{\mathtt{top}} }a_{\vec{v}}
			e^{\mathbf{i}\vec{\langle\theta},\vec{v}\rangle}
			+
			\sum_{\Phi\vec{v}\neq\vec{w}_{\mathtt{top}}}a_{\vec{v}}
			\,t^{\langle \vec{r},\,\Phi\vec{v}-\vec{w}_{\mathtt{top}}\rangle}e^{\mathbf{i}\vec{\langle\theta},\vec{v}\rangle}\right|
			+\langle \vec{r},\vec{w}_{\mathtt{top}}\rangle\cdot\log t.			
		\end{eqnarray*}
		Accordingly, the integral
		$$\log V_G(t)\,=\,\frac{1}{(2\pi)^l}\int_0^{2\pi}\cdots\int_0^{2\pi} \omega(t,\vec{\theta})\,\ud \theta_1\cdots\ud \theta_l$$
		breaks into the sum of two terms. 
		The first term gives rise to the logarithmic Mahler measure
		of the Laurent polynomial
		$$q_t(z_1,\cdots,z_l)\,=\,
		\sum_{\Phi\vec{v}=\vec{w}_{\mathtt{top}}}a_{\vec{v}}
			z_1^{v_1}\cdots z_l^{v_l}
			+
			\sum_{\Phi\vec{v}\neq\vec{w}_{\mathtt{top}}}a_{\vec{v}}
			\,t^{\langle \vec{r},\,\Phi\vec{v}-\vec{w}_{\mathtt{top}}\rangle}z_1^{v_1}\cdots z_l^{v_l}.$$
		By the way $\vec{w}_{\mathtt{top}}$ is selected, as $t$ tends to $+\infty$,
		the coefficients of $q_t$ converge to those of its chief part
		$$q_{+\infty}(z_1,\cdots,z_l)\,=\,\sum_{\Phi\vec{v}=\vec{w}_{\mathtt{top}}}a_{\vec{v}}
			z_1^{v_1}\cdots z_l^{v_l}.$$
		Thus, by the continuity of Mahler measure \cite{Boyd},
		the first term of $\log V_G(t)$ approximates the logarithmic Mahler measure of $q_{+\infty}$
		as $t\to+\infty$.
		The second term is just the integration against  
		$\langle \vec{r},\vec{w}_{\mathtt{top}}\rangle\cdot\log t$, which is constant 
		with respect to $\vec{\theta}$.
		Putting together, as $t\to +\infty$,
		$$\log V_G(t)\,=\,\log M(q_{+\infty})+\langle \vec{r},\vec{w}_{\mathtt{top}}\rangle\cdot\log t+o(1).$$
		
		The calculation yields the asymptotic formula:
		$$V_G(t)\,\sim\, C_{+\infty}\cdot t^{\langle \vec{r},\vec{w}_{\mathtt{top}}\rangle}$$
		as $t\to +\infty$. The coefficient satisfies the asserted estimation
		$$C_{+\infty}\,=\,M(q_{+\infty})\,\geq\,1,$$
		because $q_{+\infty}$ is a Laurent polynomial over $\Integral$, cf.~\cite[Lemma 3.7]{Everest--Ward}.
		The same argument works for $V_G(t^{-1})$ as well, which proves the $0+$ direction. We conclude that
		$V_G(t)$ is asymptotically monomial in both ends with the coefficient 
		greater than or equal to $1$. This completes the proof.
	\end{proof}

	\begin{proof}[{Proof of Theorem \ref{rationalAsymptotic}}]
		We adopt the notations of the statement. By 
		Theorem \ref{mConvex-eBounded} and Lemma \ref{zeroOrNot}, 
		the third assumption implies that the function $V_G(t)$ is positive for all $t\in\Real_+$.
		By Lemma \ref{stable-semicontinuous}, the second condition of Lemma \ref{mConvexVersion}
		is satisfied for $V_G(t)$ and $\{V_n(t)\}_{n\in\Natural}$.
		By Lemma \ref{coefficientVA},
		the functions $\{V_n(t)\}_{n\in\Natural}$
		satisfy the first condition of Lemma \ref{mConvexVersion}.
		The third condition of Lemma \ref{mConvexVersion} has been guaranteed by
		the assumption of Theorem \ref{rationalAsymptotic}.
		Therefore, Lemma \ref{mConvexVersion} implies that $V_G(t)$ is asymptotically
		monomial in both ends with the coefficient lying in the interval $[1,V_G(1)]$.
		This completes the proof of Theorem \ref{rationalAsymptotic}.
	\end{proof}

\section{$L^2$--Alexander torsion of $3$-manifolds}\label{Sec-mainProofs}
	
	In this section, we study $L^2$--Alexander torsion of $3$-manifolds
	using the tools that we have developed.
	In subsection \ref{Subsec-efficientCellularPresentation}, 
	we recall a formula for calculation used by \cite{DFL-torsion}.
	We prove Theorem \ref{main-torsion-weak} in Subsection \ref{Subsec-degreeRFTwist},
	and Theorem \ref{main-torsion} in Subsection \ref{Subsec-degreeFullTwist}.
	
		\subsection{Efficient cellular presentation}\label{Subsec-efficientCellularPresentation}
	To calculate $L^2$--Alexander torsion of 3-manifolds, the following formula has been used
	by \cite[Proposition 9.1]{DFL-torsion}, and we state it in some more details.
	
	\begin{lemma}\label{torsionToMatrix}
		Suppose that $N$ is an irreducible orientable compact $3$-manifold 
		with empty or incompressible toral boundary.	
		There exist elements
		$u_1,v_1,\cdots, u_l,v_l\in \pi_1(N)$ and a square matrix $A$ over $\Integral\pi_1(N)$ 
		such that the following holds true. 
		The homological classes $[u_i]-[v_i]$ are nontrivial in $H_1(N;\Rational)$.
		Furthermore, for every homomorphism $\gamma:\,\pi_1(N)\to G$ 
		which induces an isomorphism under $H_1(-;\Real)$,
		and for every cohomology class $\phi\in H^1(N;\,\Real)$,
		\begin{eqnarray*}
			\tau^{(2)}(N,\gamma,\phi)(t)&\doteq& 
			\mathrm{det}^{\mathtt{r}}_{\mathcal{N}(G)}(\kappa(\gamma,\phi,t)(A))
			\cdot\prod_{i=1}^l\mathrm{det}^{\mathtt{r}}_{\mathcal{N}(G)}(\kappa(\gamma,\phi,t)(u_i-v_i))^{-1}\\
			&=&
			\mathrm{det}^{\mathtt{r}}_{\mathcal{N}(G)}(\kappa(\gamma,\phi,t)(A))
			\cdot\prod_{i=1}^l\max\{t^{\phi(u_i)},t^{\phi(v_i)}\}^{-1}.
		\end{eqnarray*}
		Moreover, given any primitive cohomology $\phi_0\in H^1(N;\,\Integral)\cong \mathrm{Hom}(\pi_1(N),\Integral)$
		in the first place, we may require in addition that $\phi_0(u_i)\neq\phi_0(v_i)$ for $i=1,\cdots,l$, and 
		that $A$ has the form:
			$$A_0+\mu\cdot\left(\begin{matrix}\mathbf{1}_{k\times k}&0\\0&0\end{matrix}\right),$$
		where	$A_0$ is a square matrix over $\Integral\mathrm{Ker}(\phi_0)$, and $\phi_0(\mu)=1$, and  
			$$k-l\,=\,x_N(\phi_0).$$
	\end{lemma}
	
	\begin{proof}
		We may assume that $H_1(N;\,\Real)$ is nontrivial since otherwise the $L^2$--Alexander torsion is constant.
		Take any primitive cohomology class $\phi_0\in H^1(N;\,\Integral)$, for example, as specified
		in the moreover part. We employ the construction of 
		S.~Friedl in \cite[Section 4]{Friedl} to produce a $\pi_1(N)$--equivariant CW complex structure
		on the universal cover of $N$. To be precise, there exist finitely many 
		properly embedded oriented compact subsurface $\Sigma_1,\cdots,\Sigma_s$ and accordingly
		$r_1,\cdots,r_s\in\Natural$, satisfying the following properties:
		\begin{itemize}
		\item $r_1[\Sigma_1]+\cdots+r_s[\Sigma_s]\in H_2(N,\partial N;\,\Integral)$ is dual to $\phi_0$.
		\item $-r_1\chi(\Sigma_1)-\cdots-r_s\chi(\Sigma_s)=x(\phi_0)$.
		\item $\Sigma_i$ are mutually disjoint and the complement of their union in $N$ is connected.
		\end{itemize}
		The calculation here is the same as \cite[Proposition 9.1]{DFL-torsion} except that instead of 
		computing square matrices induced by $\kappa(\gamma,\phi_0,t)$ there, we compute those induced by
		$\kappa(\gamma,\phi,t)$ for any class $\phi\in H^1(N;\,\Real)$. For example, the determinant
		contribution from a block
		$$\left[\begin{matrix}1&-\nu_i\\1&-z_i\end{matrix}\right]\in\Matrix_{2\times 2}(\Integral\pi_1(N)),$$
		where $i$ runs over $1,\cdots,s$
		becomes:
		\begin{eqnarray*}
		\mathrm{det}^{\mathtt{r}}_{\mathcal{N}(G)}\left(\kappa(\gamma,\phi,t)
		\left[\begin{matrix}1&-\nu_i\\1&-z_i\end{matrix}\right]\right)&=&
		\mathrm{det}^{\mathtt{r}}_{\mathcal{N}(G)}
		{\left[\begin{matrix}1&-t^{\phi(\nu_i)}\gamma(\nu_i)\\1&-t^{\phi(z_i)}\gamma(z_i)\end{matrix}\right]}\\
		&=&
		\mathrm{det}^{\mathtt{r}}_{\mathcal{N}(G)}
		{\left[\begin{matrix}1-t^{\phi(\nu_iz_i^{-1})}\gamma(\nu_iz_i^{-1})&-t^{\phi(\nu_i)}\gamma(\nu_i)\\0&-t^{\phi(z_i)}\gamma(z_i)\end{matrix}\right]}\\	
		&=&
		\mathrm{det}^{\mathtt{r}}_{\mathcal{N}(G)}\left(\kappa(\gamma,\phi,t)(z_i-\nu_i)\right).
		\end{eqnarray*}
		The elements $\nu_i$ and $z_i$ arising from Friedl's construction satisfy
		$\phi_0(\nu_i)=r_i$ and $\phi_0(z_i)=0$. Since $(\pi,\gamma,\phi_0)$ is an admissible triple
		and $\phi_0(\nu_i)-\phi_0(z_i)=r_i\neq0$,	the element $\gamma(\nu_iz_i^{-1})$ must have infinite order in $G$.
		Then \cite[Lemma 2.8]{DFL-torsion} yields
		\begin{eqnarray*}
		\mathrm{det}^{\mathtt{r}}_{\mathcal{N}(G)}\left(\kappa(\gamma,\phi,t)(z_i-\nu_i)\right)&=&
		\mathrm{det}^{\mathtt{r}}_{\mathcal{N}(G)}\left(t^{\phi(z_i)}\gamma(z_i)-t^{\phi(\nu_i)}\gamma(\nu_i)\right)\\
		&=&
		t^{\phi(z_i)}\cdot\max\{1,t^{\phi(\nu_iz_i^{-1})}\}\\
		&=&
		\max\{t^{\phi(z_i)},t^{\phi(\nu_i)}\}.
		\end{eqnarray*}
		The point here is that we do not need to require $\phi(\nu_i)-\phi(z_i)\neq0$ for all $\phi$.
		With the modification above, we see that 
		$$u_1,v_1,\cdots,u_s,v_s\in \pi_1(N)$$ can be taken to be
		$z_1,\nu_1,\cdots,z_s,\nu_s$. Similarly, we take
		$$u_{s+1},v_{s+1},\cdots,u_{2s},v_{2s}\in \pi_1(N)$$ to be
		$x_1,\nu_1,\cdots,x_s,\nu_s$ in the notations of \cite[Proposition 9.1]{DFL-torsion},
		where $\phi_0(x_i)=0$ for all $i=1,\cdots,s$.
		This gives rise to a total number of $l=2s$ pairs of $u_i$ and $v_i$.
		The matrix $A$ is a square matrix over $\Integral\pi_1(N)$ of the form
		$$\left[\begin{matrix}
		\mathbf{1}_{n_1\times n_1}&-\nu_1\cdot\mathbf{1}_{n_1\times n_1}&0&0&\cdots&0&0\\
		0&0&\ddots&\ddots&0&0&0\\
		0&\cdots&0&0&\mathbf{1}_{n_s\times n_s}&-\nu_s\cdot\mathbf{1}_{n_s\times n_s}&0\\
		*&\cdots&\cdots&*&*&*&*
		\end{matrix}\right],$$
		where $n_i=-\chi(\Sigma_i)+2$, and 
		$*$ stand for (not necessarily square) blocks with entries in $\Integral\mathrm{Ker}(\phi_0)$,
		and $\phi_0(\nu_i)=r_i$.
		
		One can further manipulate the matrix $A$ into the asserted form without 
		affecting the regular Fuglede--Kadison determinant under $\kappa(\phi,\gamma,t)$.
		This can be done by adding diagonal $\mathbf{1}_{1\times 1}$ blocks and performing
		elementary transformations using well known tricks, 
		so we omit the details, cf. \cite[Proposition 9.3]{DFL-torsion}.
	\end{proof}

\subsection{Degree for residually finite twists}\label{Subsec-degreeRFTwist}
	In this subsection, we prove Theorem \ref{main-torsion-weak}.
	Throughout this subsection, 
	let $N$ be an irreducible orientable compact $3$-manifold with empty or incompressible toral boundary,
	and $\gamma:\pi_1(N)\to G$ be a homomorphism.
	Suppose that $G$ is finitely generated and residually finite
	and $(N,\gamma)$ is weakly acyclic.
		
	For any admissible triple $(\pi_1(N),\gamma,\phi)$ over $\Real$,
	denote by
			$$\tau^{(2)}(N,\gamma,\phi):\,\Real_+\to[0,+\infty)$$
	any representative of the associated $L^2$--Alexander torsion.
	
	\begin{lemma}\label{nonzeroTorsion}
		Given any admissible triple $(\pi_1(N),\gamma,\phi)$ over $\Real$,
			$$\tau^{(2)}(N,\gamma,\phi)(1)\,>\,0.$$
	\end{lemma}
	
	\begin{proof}
		As $(N,\gamma)$ is weakly acyclic, 
		it follows from the definition that
		$\tau^{(2)}(N,\gamma,\phi)(1)$ is the $L^2$--torsion of the pair $(N,\gamma)$,
		namely, the $L^2$--torsion of the covering space of $N$ which corresponds to $\mathrm{Ker}(\gamma)$
		equipped with the action of $\mathrm{Im}(\gamma)$. The latter can be computed
		through a weakly acyclic Hilbert chain complex of which the boundary operators
		are represented by matrices over $\Integral\mathrm{Im}(\gamma)$. As $G$ is residually finite,
		\cite[Theorem 3.4 (2)]{Lueck-approximating} implies that
		$\tau^{(2)}(N,\gamma)$ is a multiplicatively alternating
		product of positive constants which are no smaller than $1$, hence must be nonzero.
	\end{proof}
	
	\begin{lemma}\label{degree-bTorsion}
		Let $u_1,v_1,\cdots, u_l,v_l\in \pi_1(N)$ be a collection of elements
		and $A$ be  a square matrix over $\Integral\pi_1(N)$ 
		as asserted by Lemma \ref{torsionToMatrix}. 
		Given any admissible triple $(\pi_1(N),\gamma,\phi)$ over $\Real$, 
		the following formula holds valid and true:
		$$\mathrm{deg}^{\mathtt{b}}(\tau^{(2)}(N,\gamma,\phi))\,=\,
		\mathrm{deg}^{\mathtt{b}}\left(\mathrm{det}^{\mathtt{r}}_{\mathcal{N}(G)}(\kappa(\phi,\gamma,t)(A))\right)-\sum_{i=1}^l|\phi(u_i)-\phi(v_i)|.$$		
	\end{lemma}
	
	\begin{proof}
		The function $\mathrm{det}^{\mathtt{r}}_{\mathcal{N}(G)}(\kappa(\phi,\gamma,t)(A))$ of $t\in\Real_+$
		is multiplicatively convex by Theorem \ref{mConvex-eBounded}. In fact,
		it is nowhere zero and hence with bounded exponent by Lemmas \ref{nonzeroTorsion}, \ref{torsionToMatrix},
		and \ref{zeroOrNot}. Thus it is valid to speak of
		$\mathrm{deg}^{\mathtt{b}}(\tau^{(2)}(N,\gamma,\phi))$ and the formula
		follows immediately from Lemma \ref{torsionToMatrix}.
	\end{proof}
	
	\begin{proof}[{Proof of Theorem \ref{main-torsion-weak}}]
		We continue to adopt the assumptions of this subsection.
		It follows from Lemmas \ref{torsionToMatrix}, \ref{nonzeroTorsion}, and Theorem \ref{mConvex-eBounded}
		that $\tau^{(2)}(N,\gamma,\phi)$ is everywhere positive and continuous in $t\in\Real_+$.
		For any constants $a,b\in\Real$, note that the function $\max\{t^a,t^b\}^{-1}$ can
		always be turned into a multiplicatively convex function
		by multiplying a sufficiently high power of $\max\{1,t\}$, for example, by making the power
		at least $|a-b|$. It further follows that 
		$\tau^{(2)}(N,\phi)\cdot\max\{1,t\}^m$ is multiplicatively convex
		with bounded exponent any sufficiently large positive constant $m$.
		The Lipschitz continuity of $\mathrm{deg}^{\mathtt{b}}(\tau^{(2)}(N,\gamma,\phi+\gamma^*\xi))$ 
		as a function of $\xi\in H^1(G;\Real)$
		is a consequence of Theorem \ref{continuityOfDegree}.
		Therefore, it remains to show that for all admissible triple $(N,\gamma,\phi)$,
		the following comparison holds true:
			$$\mathrm{deg}^{\mathtt{b}}(\tau^{(2)}(N,\gamma,\phi))\,\leq\,x_N(\phi).$$
		
		To this end, we first prove the comparison 
		for any admissible triple $(N,\gamma,\phi_0)$ where $\phi_0$ is a primitive class
		in $H^1(N;\Integral)$. Let $u_1,v_1,\cdots, u_l,v_l\in \pi_1(N)$ be a collection of elements
		and $A$ be a square matrix over $\Integral\pi_1(N)$ 
		as guaranteed by the `moreover' part of Lemma \ref{torsionToMatrix}. 
		It is clear that
		for any arbitrary $\delta>0$,
			$$\lim_{t\to0+}\mathrm{det}^{\mathtt{r}}_{\mathcal{N}(G)}(\kappa(\gamma,\phi_0,t)(A))\cdot t^{\delta}\,=\,0,$$
		and
			$$\lim_{t\to+\infty}\mathrm{det}^{\mathtt{r}}_{\mathcal{N}(G)}(\kappa(\gamma,\phi_0,t)(A))\cdot t^{-k-\delta}\,=\,0,$$
		so
			$$\mathrm{deg}^{\mathtt{b}}\left(\mathrm{det}^{\mathtt{r}}_{\mathcal{N}(G)}(\kappa(\gamma,\phi_0,t)(A))\right)\,\leq\,k.$$
		On the other hand, the integrality of $\phi_0$ and the property that $\phi_0(u_i)\neq\phi_0(v_i)$ imply
			$$\sum_{i=1}^l|\phi_0(u_i)-\phi_0(v_i)|\,\geq\,l.$$
		Then Lemma \ref{degree-bTorsion} yields the comparison
		$$\mathrm{deg}^{\mathtt{b}}(\tau^{(2)}(N,\gamma,\phi_0))\,\leq\,k-l\,=\,x_N(\phi_0).$$
		
		For admissible triples over $\Rational$, the comparison follows immediately from the integral case by considering
		an integral multiple of $\phi$. For admissible triples over $\Real$, the comparison follows from the continuity
		of degree together with the continuity of Thurston norm.
		
		This completes the proof of  Theorem \ref{main-torsion-weak}.		
	\end{proof}

\subsection{Degree for the full twist}\label{Subsec-degreeFullTwist}
	In this subsection, we prove Theorem \ref{main-torsion}.
	Suppose that $N$ is an irreducible orientable compact $3$-manifold 
	with empty or incompressible toral boundary.
	When $N$ contains no hyperbolic piece in its geometric decomposition, $N$ is 
	a graph manifold, possibly a Seifert fibered space.
	Theorem \ref{main-torsion} in this case is an immediate consequence
	of \cite[Theorem 1.2]{DFL-torsion}, \cite{Herrmann}. 
	
	Therefore, throughout this section,
	we assume that $N$ contains at least one hyperbolic piece, or in other words,
	$N$ is either hyperbolic or so-called mixed. Note that $N$ is aspherical 
	so the $\ell^2$--Betti numbers of $N$ all vanish, by Lott--L\"uck \cite{Lott-Lueck}.
	For any class $\phi\in H^1(\pi_1(N);\Real)$, any representative
	of the associated full $L^2$--Alexander torsion
		$$\tau^{(2)}(N,\phi):\,\Real_+\to [0,+\infty)$$
	is everywhere positive and continuous, and 
	$\mathrm{deg}^{\mathtt{b}}(\tau^{(2)}(N,\phi))\in\Real$ 
	is at most $x_N(\phi)$,
	(Theorem \ref{main-torsion-weak}).
	It remains to determine the asymptotics as the parameter $t$ tends to $+\infty$ or $0+$.	
	
	Recall that a class $\phi\in H^1(N;\Real)$ is said to be \emph{quasi-fibered} if $\phi$ is the limit
	of a sequence of fibered classes in $H^1(N;\Rational)$.
	
	\begin{lemma}\label{quasifiberedClasses}
		Let $G$ be a finitely generated, residually finite group.
		For every homomorphism $\gamma:\pi_1(N)\to G$ which induces an isomorphism under $H_1(-;\Real)$, and
		for every quasi-fibered class $\phi\in H^1(N;\Real)$,
			$$\mathrm{deg}^{\mathtt{b}}\left(\tau^{(2)}(N,\gamma,\phi)\right)\,=\,x_N(\phi).$$
	\end{lemma}
	
	\begin{proof}
		Note that $(\pi_1(N),\gamma,\phi)$ is always
		admissible regardless of $\phi$ by Lemma \ref{homologicallyIsomorphic}. If $\phi\in H_1(N;\Rational)$ is a rational, fibered class, 
		the conclusion follows from
		\cite[Theorem 1.3]{DFL-torsion}. 
		In fact, for such $\phi$, the $L^2$--Alexander torsion
		$\tau^{(2)}(N,\gamma,\phi)$ is known to be asymptotically monomial, (indeed, eventually monomial,
		by \cite[Theorem 1.3]{DFL-torsion},) so in this case,
			$$\mathrm{deg}^{\mathrm{b}}\left(\tau^{(2)}(N,\gamma,\phi)\right)\,=\,\mathrm{deg}^{\mathrm{a}}\left(\tau^{(2)}(N,\gamma,\phi)\right)\,=\,x_N(\phi),$$
		cf. Definitions \ref{degree-a} and \ref{degree-b}.
			
		For any quasi-fibered class $\phi\in H_1(N;\Real)$, we take a sequence of 
		rational, fibered classes $\{\phi_n\}_{n\in\Natural}$ which converges to $\phi$. 
		Then by the continuity of degree (Theorem \ref{main-torsion-weak} (3)) and the formula of Lemma \ref{torsionToMatrix}, 
		we see that
		\begin{eqnarray*}
		\mathrm{deg}^{\mathtt{b}}\left(\tau^{(2)}(N,\gamma,\phi_n)\right)
		&=&\lim_{n\to\infty}\mathrm{deg}^{\mathtt{b}}\left(\tau^{(2)}(N,\gamma,\phi_n)\right)\\
		&=&\lim_{n\to\infty} x_N(\phi_n)\\
		&=&x_N(\phi).
		\end{eqnarray*}
		This completes the proof.		
	\end{proof}
	
	Let $u_1,v_1,\cdots, u_l,v_l\in \pi_1(N)$ be a collection of elements
		and $A$ be  a square matrix over $\Integral\pi_1(N)$ 
		as asserted by Lemma \ref{torsionToMatrix}.

	\begin{lemma}\label{quasifiberedTower}
		Given any class $\phi\in H^1(N;\Real)$, there exists a tower of quotients of $\pi_1(N)$
			$$\pi_1(N)\to\cdots\to \Gamma_n\to \cdots \to\Gamma_2\to\Gamma_1$$
		with all the following properties:
		\begin{itemize}
		\item The quotients $\Gamma_n$ are finitely generated and virtually abelian.
		\item The homomorphisms $\gamma_n:\pi_1(N)\to \Gamma_n$ induce
		isomorphisms under $H_1(-;\Real)$.
		\item The sequence of admissible triples $\{(\pi_1(N),\gamma_n,\phi)\}_{n\in\Natural}$ 
		forms a cofinal tower of 
		quotients of $(\pi_1(N),\gamma_\infty,\phi)$, where $\gamma_\infty$ denotes 
		$\mathrm{id}_{\pi_1(N)}:\pi_1(N)\to\pi_1(N)$.
		\end{itemize}
		Furthermore, 
		the tower can be required to satisfy:
		$$\mathrm{deg}^{\mathtt{b}}(V_n)=\mathrm{deg}^{\mathtt{b}}(V_\infty)$$
		for all $n\in\Natural$, where
		$$V_n(t)=\mathrm{det}^{\mathtt{r}}_{\mathcal{N}(\Gamma_n)}(\kappa(\phi,\gamma_n,t)(A)),$$
		and the notation $V_\infty(t)$ is understood similarly.		
	\end{lemma}
	
	\begin{proof}
		As we have assumed for this section that $N$ is either hyperbolic or mixed, there exists
		a regular finite cover $p:\tilde{N}\to N$ which corresponds to a finite index subgroup $\tilde{\pi}$ of $\pi_1(N)$,
		such that $p^*\phi\in H^1(\tilde{N};\Real)$ is quasi-fibered. This follows from a combination 
		of Agol's RFRS criterion for virtual fibering \cite{Agol-RFRS} and the virtual specialness of 
		hyperbolic and mixed $3$-manifolds \cite{Agol-VHC,Wise-book,PW-mixed},
		cf.~\cite[Subsection 10.1]{DFL-torsion}.
		Observe that for any further subgroup of finite index in $\tilde{\pi}$ which is normal in $\pi_1(N)$,
		the corresponding finite cover again carries the pull-back of $\phi$ as a quasi-fibered class.
		
		Take a cofinal tower of normal finite-index subgroups of $\pi_1(N)$,
			$$\pi_1(N)\geq \Pi_1\geq \Pi_2\geq\cdots\geq\Pi_n\geq\cdots.$$
		Possibly after intersecting the terms with $\tilde\pi$,
		we may require that $\Pi_n$ are all contained in $\tilde{\pi}$.
		For all $n\in\Natural$, define
			$$\Gamma_n\,=\,\pi_1(N)\,/\,(\mathrm{Ker}(\Pi_n\to H_1(\Pi_n;\Rational)).$$
		All the asserted properties of Lemma \ref{quasifiberedTower} hold obviously true 
		for the tower of quotients $\{\Gamma_n\}$,
		except maybe the `furthermore' part.
		
		To check the equality of degree, denote by
			$$p_n:\,\tilde{N}_n\to N$$
		the finite cover corresponding to the image of $\Pi_n$ in $\Gamma_n$.
		Taking restriction to $\pi_1(\tilde{N}_n)$ gives rise to new admissible triples $(\pi_1(\tilde{N}),\tilde\gamma_n,p^*_n\phi)$.
		By the dotted equality of Lemma \ref{torsionToMatrix}, and basic properties of regular Fuglede--Kadison
		determinants,	and Lemma \ref{quasifiberedClasses}, for all $n\in\Natural$,
		\begin{eqnarray*}
			\mathrm{deg}^{\mathtt{b}}\left(\tau^{(2)}(N,\gamma_n,\phi)\right)
			&=&\frac{1}{[\tilde{N}:N]}\cdot\mathrm{deg}^{\mathtt{b}}\left(\tau^{(2)}(\tilde{N}_n,\tilde{\gamma}_n,p^*\phi)\right)\\
			&=&\frac{1}{[\tilde{N}:N]}\cdot x_{\tilde{N}_n}(p_n^*\phi)\\
			&=&x_{N}(\phi).
		\end{eqnarray*}
		Note that the calculation above does not require the target group to be
		virtually abelian. Therefore, the same calculation for $\tau^{(2)}(N,\gamma_\infty,\phi)$
		yields the equality
		$$\mathrm{deg}^{\mathtt{b}}\left(\tau^{(2)}(N,\gamma_\infty,\phi)\right)\,=\,x_{N}(\phi).$$
		It follows from Lemma \ref{degree-bTorsion} that
		$$\mathrm{deg}^{\mathtt{b}}(V_n)=\mathrm{deg}^{\mathtt{b}}(V_\infty)$$
		for all $n\in\Natural$.
	\end{proof}
	
	\begin{proof}[{Proof of Theorem \ref{main-torsion}}]
	We continue to adopt the assumptions of this subsection.
	It suffices to prove the statements (2), (3), and (4).
	
	Given $N$ hyperbolic or mixed and any $\phi\in H^1(N;\Real)$,
	we take a tower of quotients as guaranteed by Lemma \ref{quasifiberedTower}.
	By Theorem \ref{rationalAsymptotic}, we see that the function (now dropping the subscript $\infty$)
		$$V(t)=\mathrm{det}^{\mathtt{r}}_{\mathcal{N}(\Gamma_n)}(\kappa(\phi,\mathrm{id}_{\pi_1(N)},t)(A))$$
	is asymptotically monomial in both ends.
	In fact, as $t\to+\infty$,
		$$V(t)\sim C_{+\infty}\cdot t^{\mathrm{deg}^{\mathtt{b}}_{+\infty}(V)}$$
	for some constant
		$$C_{+\infty}\in\left[1,e^{\mathrm{Vol}(N)/6\pi}\right],$$
	and the same statement holds true with $+\infty$ replaced by $0+$.
	Here the upper bound comes from
		$$V(1)\,=\,\tau^{(2)}(N,\phi)(1)\,=\,\tau^{(2)}(N)\,=\,e^{\mathrm{Vol}(N)/6\pi}.$$
	Therefore, $\tau^{(2)}(N,\phi)$ is also asymptotically monomial in both ends with the same
	estimation of coefficients.
	In particular, the asymptote degree of $\tau^{(2)}(N,\phi)$ is valid,
	and 
		$$\mathrm{deg}^{\mathtt{a}}\left(\tau^{(2)}(N,\phi)\right)\,=\,\mathrm{deg}^{\mathtt{b}}\left(\tau^{(2)}(N,\phi)\right)\,=\,x_N(\phi).$$
	By the symmetry of $L^2$--Alexander torsion for $3$-manifolds \cite{DFL-symmetric},
	we further imply
		$$C_{+\infty}\,=\,C_{0+}.$$
	This allows us to refer to both of them by one notation:
		$$C(N,\phi)\in\left[1,e^{\mathrm{Vol}(N)/6\pi}\right].$$

	It remains to argue that $C(N,\phi)$ depends upper semi-continuously on $\phi\in H^1(N;\Real)$. In fact,
	suppose that $\{\phi_n\in H^1(N;\Real)\}_{n\in\Natural}$ is a sequence of cohomology classes which converges to $\phi$.
	We write
		$$V(\phi_n,t)\,=\,\mathrm{det}^{\mathtt{r}}_{\mathcal{N}(\Gamma_n)}(\kappa(\phi_n,\mathrm{id}_{\pi_1(N)},t)(A)).$$
	By Lemma \ref{norm-semicontinuous}, for all $t\in\Real_+$,
		$$\limsup_{n\to\infty} V(\phi_n,t)\leq V(t).$$
	By the continuity of degree (Theorem \ref{continuityOfDegree}),
		$$\lim_{n\to\infty} \mathrm{deg}^{\mathtt{b}}\left(V(\phi_n,t)\right)=\mathrm{deg}^{\mathtt{b}}\left(V(t)\right)$$
	Then it follows from Lemma \ref{mConvexVersion} that
		$$C(N,\phi)\,\geq\limsup_{n\to\infty} C(N,\phi_n).$$
	In other words, the leading coefficient $C(N,\phi)$ is upper semicontinuous as a function of $\phi\in H^1(N;\Real)$.

	This completes the proof of Theorem \ref{main-torsion}.
	\end{proof}

\section{Example}\label{Sec-example}
	We conclude our discussion with an example regarding nontrivial leading coefficients.
	Specifically, we construct an oriented closed $3$-manifold $N$ such that the leading coefficient
	$C(N,\phi)$ of the full $L^2$--Alexander torsion $\tau^{(2)}(N,\phi)$ gives rise to values other than the asserted
	bounds, as $\phi$ varies over $H^1(N;\Real)$.
	
	The oriented closed $3$-manifold
		$$N\,=\,K\cup\bigcup_{i\in\Integral/3\Integral} J_i$$
	is constructed by gluing a product piece $K$ and three figure-eight knot complements $J_i$ as follows. Let 
		$$K\cong \Sigma_{0,3}\times S^1$$
	be the product of the thrice holed sphere and the circle.
	We mark the boundary components of $\Sigma_{0,3}$ in cyclic order.
	For each $i\in\Integral/3\Integral$, denote by $\partial_i K\cong \partial_i \Sigma_{0,3}\times S^1$ 
	the $i$-th boundary component of $K$ accordingly. For each $i\in\Integral/3\Integral$,
	take a copy of a figure-eight knot complement 
		$$J_i\cong S^3\setminus\mathrm{Nhd}^\circ(\mathbf{4}_1).$$
	We remind the reader that
	the interior of the figure-eight complement $J_i$ is a punctured torus bundle over the circle with a pseudo-Anosov monodromy,
	and it has a unique complete hyperbolic structure of volume $\mathrm{Vol}(J_i)\,=\,2v_3$, where $v_3\approx1.01494$ is the volume of
	the regular ideal hyperbolic tetrahedron.
	Denote by $\mu_i$ and $\lambda_i$ the longitude and the meridian of $J_i$ accordingly,
	so that the boundary of $J_i$ has a canonical product structure $\partial J_i\cong \lambda_i\times \mu_i$.
	Endow $K$ and $J_i$ with canonical orientations so that the boundary is oriented accordingly.
	The oriented closed $3$-manifold $N$ is obtained 
	by gluing $K$ and $J_i$ along the boundary in such a way 
	that $\partial_i K$ is identified with $-\partial J_i$ via an isomorphism that
	takes the factor $\partial_i \Sigma_{0,3}$ to $\lambda_i$ and the factor $S^1$ to $-\mu_i$.
	
	Note that the inclusion maps induce an embedding
	\begin{eqnarray*}
		H^1(N;\Real)&\to& H^1(J_0;\Real)\oplus H^1(J_1;\Real)\oplus H^1(J_2;\Real)\\
		\phi&\mapsto&(\phi_0,\phi_1,\phi_2)
	\end{eqnarray*}
	By identifying $H^1(J_i;\Real)$ with $\Real$,
	we can identify $H^1(N;\Real)$ with the $2$-subspace of the $3$-space given by the linear equation:
		$$\phi_0+\phi_1+\phi_2\,=\,0.$$
	By the fibration structure of the figure-eight complement,
	it is easy to argue topologically that the Thurston norm of any cohomology class $\phi$ in $H^1(N;\Real)$
	is given by the formula:
		$$x_N(\phi)=|\phi_0|+|\phi_1|+|\phi_2|.$$
	The unit ball $B_x(N)$ of $x_N$ is hence the region bounded by the regular hexagon whose vertices are
	$(\pm\frac12,\mp\frac12,0)$, $(0,\pm\frac12,\mp\frac12)$,	and $(\mp\frac12,0,\pm\frac12)$.
	There are no fibered cones because the restriction of every primitive class $\phi\in H^1(N;\Integral)$ to $K$ vanishes
	on the Seifert fiber $[S^1]\in H_1(K;\Integral)$, which means 
	no subsurface that is dual to $\phi$ 
	could be transverse to the Seifert fibration everywhere (or so-called horizontal) restricted to $K$.
	
	The full $L^2$--Alexander torsion of $N$ associated with any cohomology class $\phi\in H^1(N;\Real)$
	can calculated by the formula:
		$$\tau^{(2)}(N,\phi)\doteq\tau^{(2)}(J_0,\phi_0)\cdot\tau^{(2)}(J_1,\phi_1)\cdot\tau^{(2)}(J_2,\phi_2).$$
	This follows from \cite[Theorem 3.35 (1)]{Lueck-book}, (see \cite[Theorem 3.93 (2)]{Lueck-book} for a similar calculation).
	Note that in our case, the pieces $K$ and $J_i$ are weakly acyclic
	glued along tori which contribute nothing to the $L^2$--torsion of the twisted chain complex.
	There ought to be a factor $\tau^{(2)}(K,\phi_K)$ 
	corresponding to the restriction of $\phi$ to $K$
	on the right-hand side, but that factor is represented by
	$1$ according to \cite{Herrmann}, cf.~\cite[Theorem 1.2]{DFL-torsion}.
	For each $i\in\Integral/3\Integral$, it follows from the fiberedness of the figure-eight knot complement
	that the leading coefficient
		$$C(J_i,\phi_i)\,=\,\begin{cases}e^{v_3/3\pi}&\phi_i=0\\1&\phi_i\neq0\end{cases}$$
	Therefore, for any cohomology class $\phi=(\phi_1,\phi_2,\phi_3)\in H^1(N;\Real)$, 
	the leading coefficient of $\tau^{(2)}(N,\phi)$ is given by the formula:
		$$C(N,\phi)\,=\,e^{\frac{\delta(\phi)\cdot v_3}{3\pi}}$$
	where $\delta(\phi)$ denotes the number of zero coordinates in $(\phi_0,\phi_1,\phi_2)$ 
	subject to the constraint	$\phi_0+\phi_1+\phi_2=0$.
	To summarize, the leading coefficient $C(N,\phi)$ equals $e^{\mathrm{Vol}(N)/6\pi}$ at the origin,
	and $e^{\mathrm{Vol}(N)/18\pi}$ along the six radial rays through the vertices of $B_x(N)$
	(except at the origin),
	and $1$ in the rest part of $H^1(N;\Real)$.

\bibliographystyle{amsalpha}


\end{document}